\documentclass[11pt]{amsart}

\numberwithin{equation}{section}
\numberwithin{figure}{section}

\usepackage{tikz-cd}
\usepackage{amssymb}
\usepackage[T1]{fontenc}
\usepackage{lmodern}
\usepackage{mathtools}
\usepackage[headings]{fullpage}
\usepackage{braket}

\usepackage{color}
\usepackage[normalem]{ulem}
\usepackage[colorlinks=true,linkcolor=blue, citecolor=blue, urlcolor=blue,
pagebackref=true]{hyperref}
\usepackage{enumitem}

\usepackage{tikz}
\usepackage{xparse}

\newcount\tableauRow
\newcount\tableauCol

\makeatletter
\def\theenumi{\@alph\c@enumi}
\DeclareMathSizes{\@xipt}{\@xipt}{7}{6}
\makeatother

\usepackage{units}
\usepackage{ytableau}

\theoremstyle{plain}
\newtheorem{theorem}[equation]{Theorem}
\newtheorem{lemma}[equation]{Lemma}
\newtheorem{corollary}[equation]{Corollary}
\newtheorem{proposition}[equation]{Proposition}
\theoremstyle{definition}

\newtheorem{remark}[equation]{Remark}

\newtheorem{example}[equation]{Example}

\newtheorem{definition}[equation]{Definition}

\newtheorem{notation}[equation]{Notation}

\newtheorem{discussion}[equation]{Discussion}

\newtheorem{observation}[equation]{Observation}

\newtheorem{construction}[equation]{Construction}

\newcommand\spann[1]{\left\langle #1\right\rangle}
\newcommand{\MC}{\mathit{M}\kern -0.1em\mathrm{1}}
\newcommand{\MCC}{\mathit{M}\kern -0.1em\mathrm{2}}
\newcommand{\Grass}{\mathrm{G}}
\newcommand{\Std}{\mathrm{Std}}

\newcommand{\pl}{\mathrm{p}}
\newcommand{\Bl}{\mathrm{Block}}
\newcommand{\Cl}{\mathrm{Class}}
\newcommand{\He}{\mathrm{Head}}
\newcommand{\GL}{\mathrm{GL}}

\newcommand{\emptypart}{\emptyset}
\newcommand\wmodel[2]{v_{#1}^{#2}}
\newcommand\dprl[1]{b_{#1}}
\newcommand{\tab}{\mathbb{T}}

\title[Levi Subgroup Actions on Schubert Varieties]{Levi Subgroup Actions on Schubert Varieties, Induced Decompositions of their Coordinate Rings, and Sphericity Consequences}
\author{Reuven Hodges}
\address{Northeastern University, Boston, Massachusetts. USA.}
\email{hodges.r@husky.neu.edu}

\author{Venkatramani Lakshmibai}
\address{Northeastern University, Boston, Massachusetts. USA.}
\email{lakshmibai@neu.edu}

\begin{document}

\begin{abstract}
 Let $L_w$ be the Levi part of the stabilizer $Q_w$ in $GL_N$ (for left multiplication) of a Schubert variety $X(w)$ in the Grassmannian $\Grass_{d,N}$. For the natural action of $L_w$ on $\mathbb{C}[X(w)]$, the homogeneous coordinate ring of $X(w)$ (for the Pl\"ucker embedding), we give a combinatorial description of the decomposition of $\mathbb{C}[X(w)]$ into irreducible $L_w$-modules; in fact, our description holds more generally for the action of the Levi part $L$ of any parabolic subgroup $Q$ that is contained in $Q_w$. This decomposition is then used to show that all smooth Schubert varieties, all determinantal Schubert varieties, and all Schubert varieties in $\Grass_{2,N}$ are spherical $L_w$-varieties. \end{abstract}
\maketitle

\section{Introduction}
\label{sec:intro}
\numberwithin{equation}{section}
Let $\GL_N$ be the group of invertible $N \times N$ matrices over $\mathbb{C}$. Let $B$ be the Borel subgroup of $\GL_N$ consisting of upper triangular matrices, and $T$ the maximal torus consisting of diagonal matrices. For $1\le d\le N-1$, let $\Grass_{d,N}$ denote the Grassmannian variety consisting of $d$-dimensional subspaces of $\mathbb{C}^N$. 
For the canonical action of $\GL_N$ on $\Grass_{d,N}$ given by left multiplication, the $T$-fixed points in $\Grass_{d,N}$ are denoted by $[e_w]$ for $w \in I_{d,N}:=\{(i_1,\ldots,i_d)\mid 1\leq i_1 < \cdots < i_d \leq N\}$. The Schubert varieties in $\Grass_{d,N}$ are the Zariski closures of $B$-orbits through the $T$-fixed points, with the canonical reduced scheme structure. That is for $w \in I_{d,N}$ the Schubert variety $X(w):=\overline{B[e_w]}$. There is a natural partial order on $I_{d,N}$, referred to as the \emph{Bruhat order}, induced by the partial order on the set of Schubert varieties given by inclusion. We denote the homogeneous coordinate ring of $X(w)$ for the Pl\"ucker embedding by $\mathbb{C}[X(w)]$. 

Let $R$ be a connected reductive group with $B_R$ a Borel subgroup. Suppose that $X$ is an irreducible $R$-variety, then $X$ is a spherical $R$-variety if it is normal  and has an open dense $B_R$-orbit.

Our initial goal was to understand when a Schubert variety $X(w)$ is spherical for the left multiplication action of reductive subgroups of $\GL_N$ that stabilize $X(w)$. Using Proposition \ref{proposition:coneProjVarMultFreeSphericity} we relate this sphericity question to the module structure of $\mathbb{C}[X(w)]$ under the induced action of these reductive subgroups.

Fix a $w\in I_{d,N}$. There is a canonical choice of reductive subgroups of $\GL_N$ that stabilize $X(w)$. Let $Q_w$ be the stabilizer in $\GL_N$ of $X(w)$; it is clearly a parabolic subgroup of $\GL_N$. Let $L_w$ be the Levi part of $Q_w$, it is a reductive group. We have a natural action of $L_w$ on $\mathbb{C}[X(w)]$. The main result of this paper is an explicit description of the decomposition of $\mathbb{C}[X(w)]$ into irreducible $L_w$-submodules. In fact, our description holds for a much larger class of reductive subgroups, that is, for the Levi part $L$ of any parabolic subgroup $Q \subseteq Q_w$(cf. Theorem \ref{theorem:mainDecompositionTheorem} and Corollary \ref{corollary:mainDecompositionTheoremIrred}). Though it would be enough to give such a decomposition for $L_w$ and deduce the result for any $L \subseteq L_w$ by using branching rules, our procedure is independent of the choice of $L$; further, using our description for any $L$ we are able to deduce branching rule formulas (cf. Remark \ref{remark:howeBranching}).

Our decomposition proof uses the standard monomial basis for $\mathbb{C}[X(w)]$ (cf. Theorem \ref{theorem:stdMonTheoryFundamental}). As a graded ring, we have that $\mathbb{C}[X(w)]_r, r\in\mathbb{N}$ has a vector space basis given by the set $\Std_r$ of all standard monomials of degree $r$. Observe that $\Std_1$ is precisely the set of Pl\"{u}cker coordinates \{$\pl_{\tau} \mid \tau \in I_{d,N},\;\tau \leq w\}$ (cf. Section \ref{subsec:stdMonThGrass}). We give the decomposition for $\mathbb{C}[X(w)]_r$ as an $L$-module, which we describe briefly below, in terms of certain Weyl modules associated to $L$. 

Given $X(w)$, and a Levi subgroup $L$ as above, our first main step involves capturing certain Schubert subvarieties $X(\theta)$, characterized by the property that $L$ is contained in the Levi part of $Q_{\theta}$, the stabilizer in $\GL_N$ of $X(\theta)$. A combinatorial description of all such $\theta \in H_w := \left\{\tau \in I_{d,N} \mid \tau \leq w\right\}$ is given in Proposition \ref{proposition:HeadTypeW}. We refer to these $\theta$ as the heads of type $L$ and denote the subset of heads of type $L$ in $H_w$ by $\He_L$. The critical part of this step is showing how $L$ gives rise to a nice partition of the Hasse diagram of $H_w$ into disjoint subdiagrams, each containing a unique head of type $L$. Then for a $\tau \in H_w$ we define $\theta_{\tau} \in \He_L$ to be the unique head in the disjoint subdiagram containing $\tau$.
Finally for $\theta \in \He_L$ we define $\Std_{\theta}:=\{\pl_{\tau} \in \Std_1 \mid \theta_{\tau}=\theta \}$. This gives us the decompositions 

\begin{center}
$\displaystyle \Std_1=\bigsqcup_{\theta \in \He_L} \Std_{\theta}$
\end{center}
and
\begin{center}
$\displaystyle \mathbb{C}[X(w)]_1 = \spann{\Std_1} = \bigoplus_{\theta \in \He_L} \spann{\Std_{\theta}}$,
\end{center}
where $\spann{Y}$ denotes the linear span of the elements in $Y \subseteq \mathbb{C}[X(w)]$.

This decomposition of $\mathbb{C}[X(w)]_1$ would hold for any partition of the Hasse diagram of $H_w$ into disjoint subdiagrams each containing a unique head of type $L$. However, for the partition we consider we have that the $\spann{\Std_{\theta}}$ are in fact  irreducible $L$-submodules; our partition is the unique partition for which this is the case. Thus the above decomposition is an $L$-module decomposition of the degree one part of $\mathbb{C}[X(w)]$.

The second step is to extend this idea to higher degrees. For $\tau_1,\ldots,\tau_r \in H_w$ we define the degree r head of $(\tau_1,\ldots,\tau_r)$ to be the sequence $(\theta_{\tau_1},\ldots,\theta_{\tau_r})$. We define $\He_{L,r}$ to be the set of all degree r heads. Then $\He_{L,r}^{std}$ is defined to be the subset of degree r heads $(\theta_1,\ldots,\theta_r)$ such that $\theta_1 \geq \cdots \geq \theta_r$. As before for $\underline{\theta} \in \He_{L,r}^{std}$ we define $\Std_{\underline{\theta}}:=\{\pl_{\tau_1}\cdots \pl_{\tau_r} \in \Std_r \mid \underline{\theta} = (\theta_{\tau_1},\ldots,\theta_{\tau_r})\}$. Remarkably these $\Std_{\underline{\theta}}$ for $\underline{\theta} \in \He_{L,r}^{std}$ once again partition the set $\Std_r$. This is by no means readily apparent and is due to the fact that given $\tau_1,\ldots,\tau_r \in H_w$ such that $\tau_1 \geq \cdots \geq \tau_r$ we have $\theta_{\tau_1} \geq \cdots  \geq \theta_{\tau_r}$(cf. Proposition \ref{proposition:PartitionofDegRStandard}). Thus we have the decompositions
\begin{center}
$\displaystyle \Std_r=\bigsqcup_{\underline{\theta} \in \He_{L,r}^{std}} \Std_{\underline{\theta}}$
\end{center}
and
\begin{center}
$\displaystyle \mathbb{C}[X(w)]_r = \spann{\Std_r} = \bigoplus_{\underline{\theta} \in \He_{L,r}^{std}} \spann{\Std_{\underline{\theta}}}$.
\end{center}
Unfortunately, for $r>1$ the latter decomposition is no longer a $L$-module decomposition. This is due to the way that the $L$-action interacts with the standard monomial straightening rule which results in certain $\spann{\Std_{\underline{\theta}}}$ not being $L$-stable. Thus in higher degrees our decomposition must be modified.

To achieve this we introduce the partial order $\geq_{str}$ on the set of degree r heads $\He_{L,r}$. This order is just the lexicographic order on these sequences, with the order on each individual entry being the Bruhat order. For $\underline{\theta} \in \He_{L,r}^{std}$ we define 
\begin{center}
$\Std_{\underline{\theta}}^{\geq_{str}} = \left\{\pl_{\tau_1} \cdots \pl_{\tau_r} \in \Std_r \mid (\theta_{\tau_1},\ldots,\theta_{\tau_r}) \geq_{str} \underline{\theta} \right\}$
\end{center}
and 
\begin{center}
$\Std_{\underline{\theta}}^{>_{str}} = \left\{\pl_{\tau_1} \cdots \pl_{\tau_r} \in \Std_r \mid (\theta_{\tau_1},\ldots,\theta_{\tau_r}) >_{str} \underline{\theta} \right\}$.
\end{center}
The next step is to show that $\spann{\Std_{\underline{\theta}}^{\geq_{str}}}$ and $\spann{\Std_{\underline{\theta}}^{>_{str}}}$ are $L$-stable. To show this we equivalently check that they are $Lie(L)$ stable. Once we have done this we define $U_{\underline{\theta}}$ to be a $L$-module complement of $\spann{\Std_{\underline{\theta}}^{>_{str}}}$ inside of $\spann{\Std_{\underline{\theta}}^{\geq_{str}}}$. 

A final result is required before the description of the degree r decomposition. The $U_{\underline{\theta}}$ are $L$-submodules of  $\spann{\Std_r}$. As $L$ is equal to a product of general linear groups we should be able to describe any $L$-module, in particular the $U_{\underline{\theta}}$, in terms of tensor products of Weyl modules. As vector spaces $U_{\underline{\theta}} \cong \spann{\Std_{\underline{\theta}}}$. In view of this we first describe a vector space map from $\spann{\Std_{\underline{\theta}}}$ to a certain external tensor product of (skew) Weyl modules denoted $\mathbb{W}_{\underline{\theta}}$, using the combinatorics of the standard monomials. This map as well as some character arguments are then used to conclude that the $L$-module $U_{\underline{\theta}}$ has the form $\mathbb{W}^{*}_{\underline{\theta}}$.

We are now ready to state our main result(cf. Theorem \ref{theorem:mainDecompositionTheorem}).

\begin{theorem}
Let $\underline{\theta} \in \He_{L,r}^{std}$. There exists a $L$-submodule $U_{\underline{\theta}} \subseteq \spann{\Std_r}$ such that we have the following $L$-module isomorphisms:
\begin{enumerate}
\item $\spann{\Std_{\underline{\theta}}^{\geq_{str}}} = U_{\underline{\theta}} \oplus \spann{\Std_{\underline{\theta}}^{>_{str}}}$.
\item $\displaystyle \spann{\Std_r} = \bigoplus_{\underline{\theta} \in \He_{L,r}^{std}} U_{\underline{\theta}}$.
\item $U_{\underline{\theta}} \cong \mathbb{W}^{*}_{\underline{\theta}}$
\end{enumerate}
\end{theorem}

The $U_{\underline{\theta}}$ may not be irreducible $L$-modules, but their decomposition into irreducibles can now be computed simply by calculating the decomposition of certain tensor products of Weyl modules. This is done in Corollary \ref{corollary:mainDecompositionTheoremIrred} where we give the explicit decomposition of $\mathbb{C}[X(w)]$ into irreducible $L$-modules.

As an important application, we note that if $\mathbb{C}[X(w)]$ has a multiplicity free decomposition into irreducible $L_w$-modules then $\widehat{X(w)}$, the cone over $X(w)$, is a spherical $L_w$-variety. We then use this to show that $X(w)$ is a spherical $L_w$-variety. We conclude that all smooth Schubert varieties, all determinental Schubert varieties(and determinental varieties), and all Schubert varieties in $\Grass_{2,N}$ are spherical $L_w$-varieties. We also get that the coordinate ring of any determinental variety has a multiplicity free decomposition into irreducible $L_w$-modules.

We hope to extend the results of this paper, using similar techniques, to any Schubert variety in $\GL_N/Q$, where $Q$ is any parabolic subgroup, as well as to Schubert varieties in the Lagrangian and Orthogonal Grassmannians. The combinatorial results that one may obtain for the spherical Schubert varieties (by virtue of them being spherical varieties) is also a topic we hope to investigate.
 
 The sections are organized as follows: Section 2 is on Preliminaries pertaining to Schubert varieties in $\Grass_{d,N}$, standard monomial basis, and representation theory of the general linear group. In Section 3, we introduce the heads of type $L$, the degree r heads, and after proving some preparatory results, we determine the decomposition (as an $L$-module) of $\mathbb{C}[X(w)]$. In Section 4, we show that $X(w)$ is a spherical $L_w$-variety if the decomposition of $\mathbb{C}[X(w)]$ into irreducible $L_w$-submodules is multiplicity-free. Using this, and the results from section 3, we prove that many classes of Schubert varieties are spherical $L_w$-varieties.

\section*{Acknowledgements}
The authors thank Guido Pezzini for many helpful discussions during their visit to FAU Erlangen-N\"{u}rnberg in June 2015 and for his insights into spherical varieties in particular. The authors would also like to thank Friedrich Knop for some useful discussions on spherical varieties. Additionally, the authors would like to thank FAU Erlangen-N\"{u}rnberg for the hospitality extended to them during their visit in June 2015. Finally, the authors thank the referee for many helpful comments and suggestions. 

\section{Preliminaries}
\label{sec:prelim}
\subsection{The general linear group}
\label{subsec:glN}
\numberwithin{equation}{subsection}
In this section and the next we fix the notation that will be used throughout the paper. For a more in depth introduction to these topics see \cite{MR3408060} and \cite{MR2388163}.

Fix a positive integer $N$, and let $\{e_1,\ldots,e_N\}$ be the standard basis of $\mathbb{C}^{N}$. We will do all computations over the field $\mathbb{C}$. We denote by $\GL_N$ the invertible $N \times N$ matrices over $\mathbb{C}$. Let $T$ be the \emph{standard maximal torus} comprised of diagonal matrices, and $B$ the \emph{standard Borel subgroup} comprised of upper triangular matrices.

Let $\mathfrak{X}(T):={\mathrm{Hom}}_{{\mathrm{alg.gp}}}(T,{\mathbb{C}}^{*})$ be the character group of $T$; it is a free abelian group of rank $N$ with a basis $\{\epsilon_i, 1\le i\le N\}, \epsilon_i$ being the character which sends a diagonal matrix $diag(t_1,\ldots, t_N)$ to its $i$-th entry $t_i$. The elements of $\mathfrak{X}(T)$ will be referred to (formally) as {\emph weights}. We will often simplify our notation by referring to an element of $\mathfrak{X}(T)$ by the sequence $(a_1,\ldots,a_N)$, $a_i \in \mathbb{Z}$, which corresponds to the weight $\sum\,a_i\epsilon_i \in \mathfrak{X}(T)$ that sends $diag(t_1,\ldots, t_N)$ to $\prod_i t_i^{a_i}$. 

We will call a weight $(a_1,\ldots,a_N)$ such that $a_1\ge\cdots\ge a_N$ a {\emph weakly dominant weight}. If in addition, $a_N \geq 0$ then it is called a dominant weight (cf. \cite{MR1153249}). Recall that the set of all dominant weights gives an indexing of the set of all irreducible polynomial representations of $GL_N$, while the weakly dominant weights index the set of all irreducible rational representations of $GL_N$. 

Let $V$ be a finite-dimensional $T$-module. Then we have the decomposition 
\begin{center}
$V=\displaystyle \bigoplus_{\chi\in \mathfrak{X}(T)}\,V_\chi$
\end{center}
where $V_\chi$ is $T$-weight space consisting of all vectors $v\in V$ such that $t\cdot v=\chi(t)v$, for all $t\in T$. If $v \in V_\chi$ we say that $v$ has weight $\chi$, and write $wt(v)=\chi$. Let $m_\chi=dim\,V_\chi$. We define the {\emph character} of $V$, denoted char$\,(V)$, as the element in ${\mathbb Z}[\mathfrak{X}(T)]$, the group algebra of $\mathfrak{X}(T)$, given by $$char\,(V):=\sum\,m_\chi e^\chi.$$


\begin{remark}
\label{remark:dominantWeightProduct}
If $G$ is a product, say $G=GL_M \times GL_N$, by a weakly $G$-dominant weight of $G$, we shall mean a sequence $(a_1,\ldots,a_M,b_1,\ldots,b_N),a_i,b_j\in\mathbb{Z}$, with $a_1\ge\ldots\ge a_M$ and $b_1\ge \ldots\ge b_N$. If in addition, $a_M \geq 0$ and $b_N \geq 0$ then it is a $G$-dominant weight. Given a $GL_M$-module V and a $GL_N$-module $W$, consider the $GL_M\times GL_N$-module $V\otimes W$ given by the natural diagonal action; we say that $V\otimes W$ is a polynomial (rational) representation of $GL_M\times GL_N$, if $V,W$ are polynomial (rational) representations of $GL_M, GL_N$ respectively. By $char(\,V\otimes W)$, we shall mean the element $(char\,V, char\,W)$ of ${\mathbb Z}[\mathfrak{X}(T_M)]\times {\mathbb Z}[\mathfrak{X}(T_N)]$, where $T_M,T_N$ denote the maximal tori, consisting of diagonal matrices in $GL_M,GL_N$ respectively. More generally for $G=GL_{N_1} \times \cdots \times GL_{N_r}$ these notions extend in the obvious way.
\end{remark}
Let $\Phi$ be the \emph{root system} of $\GL_N$(cf. \cite[Chapter~3]{MR2388163}). It is the set $\{\epsilon_i - \epsilon_j \mid 1\leq i,j \leq N\}$, where $\epsilon_i - \epsilon_j$ is the element of $\mathfrak{X}(T)$ which sends diag$(t_1,\ldots,t_n)$ in $T$ to $t_it_j^{-1}$ in $\mathbb{C}$. Our choice of the torus $T$ and the Borel subgroup $B$ induces a set of \emph{positive roots} $\Phi^+=\{\epsilon_i - \epsilon_j \mid 1\leq i<j \leq N \}$ and \emph{simple roots} $S=\{\alpha_i:=\epsilon_i - \epsilon_{i+1} \mid 1\leq i \leq N-1 \}$. 

For every $1\leq d \leq N-1$ there is a \emph{maximal parabolic subgroup} $P_{d}$ of $\GL_N$ that corresponds to the subgroup of all matrices with a block of size $N-d \times d$ in the lower left corner with all entries equal to zero.
\begin{center}
$P_{d}=\left\{ \left[ \begin{array}{cc}
* & * \\
0_{N-d \times d} & * \\
\end{array} \right] \in \GL_N \right\}$
\end{center}
\begin{remark}
\label{remark:formofParabolics}
There is a bijection between the subsets $A\subset\{1,\ldots,N-1\}$ and the parabolic subgroups of $\GL_N$, given by 
\begin{center}
$P_A=\displaystyle \bigcap_{d \in \{1,\ldots,N-1\}\setminus A} P_{d}$.
\end{center}
\end{remark}

The Weyl group $W$ of $\GL_N$ is generated by the simple reflections $s_{\alpha_i}$ for $\alpha_i \in S$ and is isomorphic to the symmetric group $\mathfrak{S}_{N}$ of permutations on $N$ symbols under the map sending $s_{\alpha_i}$ to the transposition $(i,i+1)$. The one-line notation for elements of $W$ is a sequence $(x_1,\ldots,x_N)$ and this sequence corresponds to the permutation that sends $i \mapsto x_i$. For the parabolic subgroup $P_A$, $A\subset\{1,\ldots,N-1\}$, $W_{P_A}$ is the subgroup of $W$ generated by $\{s_{\alpha_i}\mid i \in A\}$. Since $P_{d} = P_{\{1,\ldots,d-1,d+1,\ldots,N\}}$, we have that $W_{P_{d}}=\mathfrak{S}_d \times \mathfrak{S}_{N-d}$.

\subsection{Standard monomial theory for the Grassmannian}
\label{subsec:stdMonThGrass} 
The \emph{Grassmannian} $\Grass_{d,N}$ is the set of all d-dimensional subspaces of $\mathbb{C}^{N}$. For $U\in \Grass_{d,N}$ fix a basis $\{u_1,\ldots,u_d\}$ of $U$ and define the map
\begin{center}
$\Grass_{d,N}\longrightarrow \mathbb{P}(\bigwedge^{d}\mathbb{C}^{N})$

$U \mapsto [u_1 \wedge \cdots \wedge u_d]$
\end{center}

This map does not depend on the choice of basis for $U$ and hence is well defined; this is the well-known Pl\"{u}cker embedding. Set $I_{d,N}:=\{(i_1,\ldots,i_d)\mid 1\leq i_1 < \cdots < i_d \leq N\}$. Then 
\begin{center}
$\{e_{\tau}:=e_{i_1} \wedge \cdots \wedge e_{i_d}\}_{\tau=(i_1,\ldots,i_d)\in I_{d,N}}$
\end{center}
is the standard basis for $\bigwedge^{d}\mathbb{C}^{N}$. Let $\{\pl_{\phi} \}_{\phi \in I_{d,N}}$ be the basis for $(\bigwedge^{d}\mathbb{C}^{N})^{*}$ dual to $\{e_{\tau}\}_{\tau\in I_{d,N}}$. This dual basis gives a set of projective coordinates for $\mathbb{P}(\bigwedge^{d}\mathbb{C}^{N})$, called the \emph{Pl\"{u}cker coordinates}. These coordinates have a particularly nice description for points in $\Grass_{d,N}$ in terms of determinants: for $U\in \Grass_{d,N}$ fix a basis $\{u_1,\ldots,u_d\}$ of $U$ as above and let $M$ be the $N \times d$ matrix with columns $u_1,\ldots,u_d$. Then for $\phi=(j_1,\ldots,j_d)$ the Pl\"{u}cker coordinate $\pl_{\phi}(U)$ is the $d$-minor with row indices $j_1,\ldots,j_d$ of $M$. The Pl\"{u}cker embedding equips $\Grass_{d,N}$ with a projective variety structure, realized as the zero set of the well known quadratic Pl\"{u}cker relations.

The Grassmannian is a homogeneous space for the action of $\GL_N$ induced on $\mathbb{P}(\bigwedge^{d}\mathbb{C}^{N})$ by the $\GL_N$ action on $\mathbb{C}^{N}$ and hence on $\bigwedge^{d}\mathbb{C}^{N}$. Let $e_{id}:=e_1\wedge \cdots \wedge e_d$, then $[e_{id}]\in \mathbb{P}(\bigwedge^{d}\mathbb{C}^{N})$ and the $\GL_N$ orbit through $[e_{id}]$ is $\Grass_{d,N}$, and the isotropy subgroup at $[e_{id}]$ is precisely $P_{d}$. Thus we obtain the identification  $\Grass_{d,N}=\GL_N/P_{d}$.

 The $T$-fixed points in $\GL_N/P_{d}$ are $\{wP_{d} \mid w \in W/W_{P_{d}}\}$. Since $W=\mathfrak{S}_{N}$ and $W_{P_{d}}=\mathfrak{S}_d \times \mathfrak{S}_{N-d}$ we have that $W/W_{P_{d}}$ may be identified with $I_{d,N}$. Under the identification of $\Grass_{d,N}$ with $\GL_N/P_{d}$ we see that the point $[e_{i_1}\wedge \cdots \wedge e_{i_d}]$ gets identified with $wP_{d}$ where $w$ is an element of $W$ with unordered first $d$ entries equal to $\{i_1,\ldots,i_d\}$. Thus the $T$-fixed points of $\Grass_{d,N}$ are $[e_{\tau}]$ for $\tau \in I_{d,N}$.

For $w \in I_{d,N}$ the \emph{Schubert variety} in $\Grass_{d,N}$ associated to $w$ is $X(w):=\overline{B [e_w]}$, the $B$-orbit closure of the $T$-fixed point $[e_w]$, equipped with the canonical reduced scheme structure.

There is a natural partial order on $I_{d,N}$, referred to as the \emph{Bruhat order}, induced by the partial order on the set of Schubert varieties given by inclusion. Explicitly, for $\tau:=(i_1,\ldots,i_d), w:=(\ell_1,\ldots,\ell_d) \in I_{d,N}$ we have $\tau \leq w$ if $X(\tau)\subseteq X(w)$. Since $[e_{\tau}] \in X(w)$ if and only if $X(\tau)\subseteq X(w)$, it can be shown that $\tau \leq w$ if and only if $i_1 \leq \ell_1,\ldots,i_d \leq \ell_d$.

Note that for the choice of $w=(N-d+1,\ldots,N)\in I_{d,N}$ we have that the Schubert variety $X(w)$ is in fact $\Grass_{d,N}$ itself. This follows from the \emph{Bruhat decomposition} 
\begin{center}
$X(w)=\displaystyle \bigcup_{\tau \leq w} B [e_{\tau}]$.
\end{center}

For a $d$-tuple $(x_1,\ldots,x_d)$ of integers, $(x_1,\ldots,x_d)\uparrow$ will denote the $d$-tuple obtained from $(x_1,\ldots,x_d)$ by arranging the entries in ascending order. Then for $(x_1,\ldots,x_N)$, $(y_1,\ldots,y_N) \in W$ we define
\begin{center}
$(x_1,\ldots,x_N) \leq (y_1,\ldots,y_N) \iff\; (x_1,\ldots,x_d)\uparrow \;\;\leq\; (y_1,\ldots,y_d)\uparrow \textrm{ for all } 1 \leq d \leq N-1$,
\end{center}
where the symbol $\leq$ on the right hand side is the Bruhat order defined above. This partial order on $W$ is also referred to as the Bruhat order. We define $W^{P_{d}}$ to be the set of minimal representatives of $W / W_{P_{d}}$ in $W$ under the Bruhat order. Recalling that $W_{P_{d}}=\mathfrak{S}_d \times \mathfrak{S}_{N-d}$ we see that the set of minimal representatives of $W / W_{P_{d}}$ in $W$ is
\begin{center}
$\{(x_1,\ldots,x_N) \in W \mid x_1 < \cdots < x_d\;,\;\;x_{d+1} < \ldots < x_{N} \  \}$
\end{center}
and hence $W^{P_{d}}$ can be identified with $I_{d,N}$. Thus, for the rest of the paper, elements of $W^{P_{d}}$ will be identified by their corresponding $d$-tuples in $I_{d,N}$.

\begin{remark}
\label{remark:weightPlucker}
Let $\tau\in W^{P_{d}}$, say $\tau=(i_1,\ldots,i_d)$. Denote the integers $\{1,\ldots,N\} \setminus \{i_1,\ldots,i_d\}$ by $j_1,\ldots,j_{N-d}$(arranged in ascending order). Let $det_N$ be the determinant representation of $\GL_N$ (cf. Definition \ref{definition:determinantRep}). The $\GL_N$-representation $(\bigwedge^{d}\mathbb{C}^{N})^{*}$ is isomorphic to the $\GL_N$-representation $\bigwedge^{N-d}\mathbb{C}^{N} \otimes (det_N)^{*}$, and under this identification the Pl\"{u}cker co-ordinate $\pl_\tau$ corresponds to the element $e_{j_1}\wedge\cdots\wedge e_{j_{N-d}} \otimes (e_1 \wedge \cdots \wedge e_N)^{*} \in \bigwedge^{N-d}\mathbb{C}^{N} \otimes (det_N)^{*}$. Thus, the weight of $\pl_\tau$ is $\epsilon_{j_{1}}+\cdots+ \epsilon_{j_{N-d}} + (-1,\ldots,-1)$ and is given by the sequence $\chi_{\tau}:=(\chi_1,\ldots,\chi_{N})$ where
\begin{center}
$\chi_i \coloneqq \left\{
\begin{array}{ll}
      -1 & i\in \tau\\
      0 & i\notin \tau 
\end{array} 
\right.$
\end{center}
for all $1\leq i \leq N$.
\end{remark}

Now consider the projective embedding 
\begin{center}
$X(w)\hookrightarrow \Grass_{d,N}\hookrightarrow \mathbb{P}(\bigwedge^{d}\mathbb{C}^{N})$.
\end{center}

Let $\mathbb{C}[X(w)]$ be the homogeneous coordinate ring of $X(w)$ for this projective embedding. As a $\mathbb{C}$-algebra it is generated by $\pl_{\tau}$, $\tau \leq w$. This follows from the fact that $\pl_{\tau}([e_w])=\delta_{\tau,w}$, which implies that $\pl_{\tau}|_{X(w)} \not\equiv 0$ if and only if $[e_{\tau}] \in X(w)$, which occurs if and only if $\tau \leq w$. Thus for $\tau_1,\ldots,\tau_r,w \in W^{P_{d}}$ with $\tau_i \leq w$ for all $1\leq i \leq r$, we have $\pl_{\tau_1} \cdots \pl_{\tau_r} \neq 0$ in $\mathbb{C}[X(w)]_r$.

\begin{definition} 
\label{definition:stdMonomials}
We define the monomial $\pl_{\tau_1} \cdots \pl_{\tau_r}$ to be $\emph{standard}$ if $\tau_1 \geq \cdots \geq \tau_r$. It is $\emph{standard on X(w)}$ if in addition $w\geq \tau_1$.
\end{definition}

\begin{theorem}
\label{theorem:stdMonTheoryFundamental}
(cf. \cite[Theorem~4.3.3.2]{MR2388163})Monomials of degree r standard on X(w) give a $\mathbb{C}$-basis for $\mathbb{C}[X(w)]_r$.
\end{theorem}

\subsection{The straightening algorithm}
\label{subsec:straighteningAlgorithm}

The generation portion of Theorem \ref{theorem:stdMonTheoryFundamental} usually relies on exhibiting an inductive process that takes a nonstandard monomial and writes it as a sum of standard monomials. This is called straightening the nonstandard monomial, and the entire process is referred to as the straightening process.

The straightening process on the Grassmannian is comprised of an inductive step usually referred to as a shuffle. Let $\tau:=(i_1,\ldots,i_d),\phi:=(j_1,\ldots,j_d) \in W^{P_{d}}$ with $\tau \ngeq \phi$, that is $\pl_{\tau}\pl_{\phi}$ is not standard. This implies there exists a $t$, $t\leq d$ such that $i_m \geq j_m$, for all $1\leq m \leq t-1$, and $i_{t} < j_{t}$. Let $[\tau,\phi]$ denote the set of permutations $\sigma$, other than the identity permutation, of the set $\{i_1,\ldots,i_{t},j_{t},\ldots,j_d\}$ such that $\sigma(i_1)< \cdots <\sigma(i_{t})$ and $\sigma(j_{t})< \cdots <\sigma(j_d)$. 

Define $\alpha^{\sigma} := (\sigma(i_1),\ldots,\sigma(i_t),i_{t+1},\ldots,i_d)\uparrow$ and $\beta^{\sigma} := (j_{1},\ldots,j_{t-1}, \sigma(j_t),\ldots,\sigma(j_d))\uparrow$ (recall, for a $d$-tuple $(x_1,\ldots,x_d)$, $(x_1,\ldots,x_d)\uparrow$ denotes the $d$-tuple obtained from $(x_1,\ldots,x_d)$ by arranging the entries in ascending order). Then
\begin{center}
$\pl_{\tau}\pl_{\phi} = \displaystyle \sum_{\sigma \in [\tau,\phi]} \pm \pl_{\alpha^{\sigma}}\pl_{\beta^{\sigma}}$.
\end{center}

Note that it is possible to keep track of the signs in the above summation but we omit this step since it is not needed for our consideration. It is not difficult to see that either $\pl_{\alpha^{\sigma}}=0$, due to a repeated entry, or $\alpha^{\sigma}>\tau$. For the same reasons either $\pl_{\beta^{\sigma}}=0$, due to a repeated entry, or $\beta^{\sigma}<\phi$. We will refer to this as the \emph{ordering property of the shuffle}.

A single shuffle is not always sufficient to straighten the monomial $\pl_{\tau}\pl_{\phi}$.  In this case, the shuffle is applied inductively to each monomial on the right hand side of the above sum.  This process will eventually terminate after a finite number of steps, guaranteed by the fact that there are only finitely many degree 2 monomials and the ordering property of the shuffles(cf. \cite[Chapter~4]{MR2388163}).

Finally a degree r nonstandard monomial $\pl_{\tau_1} \cdots \pl_{\tau_r} $ may be straightened by applying a shuffle to any pair $\pl_{\tau_k}, \pl_{\tau_{k+1}}$ such that $\pl_{\tau_k} \ngeq \pl_{\tau_{k+1}}$, and then continuing to apply shuffles to any resulting monomials until they are standard. Again this process will terminate in a finite number of steps.

To straighten a nonstandard monomial on $X(w)$ all that is required is to apply the straightening process for the Grassmannian and then to note that in the resulting sum of standard monomials, any that are standard but not standard on $X(w)$ are equal to zero on $X(w)$ (cf. Definition \ref{definition:stdMonomials}). 

\subsection{Young diagrams and tableaux}
\label{subsec:youngDiagramsSSYT}
This section for the most part follows the terminology of \cite{MR1464693} ,\cite{MR1153249}, and \cite{MR1676282}. Let $\lambda=(\lambda_1,\ldots,\lambda_k)$ be a collection of nonnegative integers with $\lambda_1 \geq \cdots \geq \lambda_k$, then for $|\lambda|:=\lambda_1+\cdots+\lambda_k$ we say that $\lambda$ is a \emph{partition} of $|\lambda|$. It will be useful at times to make this notation more succinct by rewriting $\lambda=(\lambda_1,\ldots,\lambda_k)$, replacing any maximal chain $\lambda_i,\ldots,\lambda_{i+j-1}$ where $\lambda_i = \cdots = \lambda_{i+j-1} = a$ with $a^j$. We identify a partition $\lambda$ with its \emph{Young diagram}, also denoted $\lambda$ for simplicity of notation, which is a collection of upper left justified boxes with $\lambda_i$ boxes in the $i$th row. These boxes are referred to by specifying row and column, with the leftmost column denoted column 1, and the topmost row denoted row 1.
\begin{example}
\label{example:youngDiagram}
The partition $(4,2,2,1)=(4,2^2,1)$ is identified with the Young diagram.
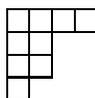
\begin{center}
\ytableausetup{smalltableaux}
\begin{ytableau}
\; & \; & \; & \; \\
\; & \; \\
\; & \; \\
\; \\
\end{ytableau}
\end{center}
\end{example}

The \emph{conjugate partition} $\lambda'$ is the partition whose diagram is the transpose of the diagram of $\lambda$, or equivalently, it is defined by setting $\lambda'_i:=\#\{j\mid \lambda_j \geq i\}$.  The conjugate partition of $(4,2,2,1)$ from Example \ref{example:youngDiagram} is $(4,3,1,1)$.

If we have a second partition $\mu$ we write $\mu \subseteq \lambda$ if the diagram for $\mu$ is contained in the diagram for $\lambda$, or equivalently, if $\mu_i \leq \lambda_i$ for $i \geq 1$. If $\mu \subseteq \lambda$ we may define the skew diagram $\lambda / \mu$ which is obtained by deleting the leftmost $\mu_i$ boxes from row $i$ of the diagram $\lambda$ for each row of $\lambda$. The number of boxes in the skew diagram is equal to $|\lambda / \mu|:=|\lambda| - |\mu|$. It is important to note here the fact that $\lambda = \lambda / \emptypart$, where $\emptypart$ is the empty partition, and so many of the definitions made for skew diagrams may be specialized to diagrams of partitions. 

\begin{example}
\label{example:skewYoungDiagram}
The skew diagram for $\lambda / \mu=(4,2,2,1)/(2,1)$ is
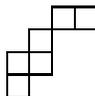
\begin{center}
\begin{ytableau}
\none & \none & \; & \; \\
\none & \; \\
\; & \; \\
\; \\
\end{ytableau}
\end{center}
\end{example}

Define $\tilde{\lambda} / \tilde{\mu}$ to be the skew diagram obtained by deleting all empty rows and columns from the skew diagram $\lambda / \mu$. The $\pi$-rotation of a skew diagram $\lambda / \mu$, written $(\lambda / \mu)^{\pi}$, is obtained by rotating $\lambda / \mu$ through $\pi$ radians. For example the $\pi$-rotation of $(4,2,2,1)/(2,1)$ is $(4,4,3,2)/(3,2,2)$.

\begin{remark}
\label{remark:pirotpartition}
If $\lambda \subseteq (m^n)$ for $m,n$ positive integers, then $(m^n)/\lambda$ is a skew diagram and $((m^n)/\lambda)^{\pi}$ is always a partition.
\end{remark}

A tableaux on $\lambda / \mu$ is an assignment of a positive integer to each box of $\lambda / \mu$. A semistandard (Young) tableaux, often abbreviated SSYT, is a tableaux where the values in each box increase weakly along each row, and increase strictly down each column. A standard (Young) tableaux is a semistandard tableaux where the values are distinct integers from 1 to the number of boxes.

\begin{example}
\label{example:skewTableauxSSYTST}
A tableaux, semistandard tableaux, and standard tableaux on $(4,2,2,1)/(2,1)$.
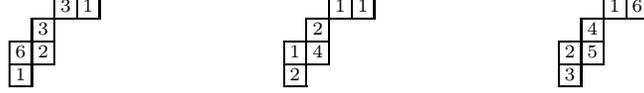
\begin{center}
\begin{ytableau}
\none & \none & 3 & 1 \\
\none & 3 \\
6 & 2 \\
1 \\
\end{ytableau}
\qquad \qquad \qquad
\begin{ytableau}
\none & \none & 1 & 1 \\
\none & 2 \\
1 & 4 \\
2 \\
\end{ytableau}
\qquad \qquad \qquad
\begin{ytableau}
\none & \none & 1 & 6 \\
\none & 4 \\
2 & 5 \\
3 \\
\end{ytableau}
\end{center}
\end{example}

If we fix a partition $\lambda$ and a bound $N$ on the maximum value that can be assigned to a box in a tableaux $\tab$ we may define the Schur function $s_{\lambda}$ by
\begin{center}
$s_{\lambda} = \displaystyle \sum_{\tab\textrm{ a SSYT on }\lambda} x_{1}^{\textrm{\# of 1's in }\tab}\cdots x_{N}^{\textrm{\# of N's in }\tab}$.
\end{center}
In the same way for a skew diagram $\lambda / \mu$ we may define the skew Schur function $s_{\lambda / \mu}$. Both the Schur functions and the skew Schur functions are symmetric functions, and the Schur functions form a vector space basis of the ring of symmetric functions in the variables $x_1,\ldots,x_N$. Thus the product of two Schur functions, which is itself a symmetric function, can be written as a sum of Schur functions
\begin{center}
$s_{\lambda}s_{\mu}=\displaystyle \sum_{\nu} c_{\lambda, \mu}^{\nu}s_{\nu}$
\end{center}
and this is one of many equivalent ways of defining the Littlewood-Richardson coefficients $c_{\lambda, \mu}^{\nu}$. Note that the above sum is over all partitions $\nu$ such that $|\nu|=|\lambda| + |\mu|$.

The Littlewood-Richardson coefficients are also critical in describing the expansion of the skew Schur functions $s_{\lambda / \mu}$ in terms of the Schur functions, namely
\begin{center}
$s_{\lambda / \mu}=\displaystyle \sum_{\nu} c_{\mu, \nu}^{\lambda}s_{\nu}$
\end{center}
and the above sum is over all partitions $\nu$ such that $|\nu|=|\lambda| - |\mu|$. In the special case where for a fixed skew Schur function $s_{\lambda / \mu}$ all the $c_{\mu, \nu}^{\lambda}$ are either 0 or 1 we say that $s_{\lambda / \mu}$ is multiplicity-free. The reason for this designation will become clear in Section \ref{subsec:schurWeylModules}.

We will need the following identity whose derivation can be found in \cite{MR1676282}.
\begin{equation}
\label{equation:skewSchurPi}s_{\lambda / \mu} = s_{(\lambda / \mu)^{\pi}}
\end{equation} 

\subsection{Schur and Weyl modules}
\label{subsec:schurWeylModules}
The skew Schur functions $s_{\lambda / \mu}$ are the characters of certain representations of $\GL_N$ for some $N \geq 1$, where the length of the first column of $\lambda$ does not exceed $N$ (cf. \cite{MR1153249}). Given a standard tableaux on the skew diagram $\lambda / \mu$ with the bound on the entries in the boxes equal to $d$, we may define two subgroups of the symmetric group $\mathfrak{S}_{d}$
\begin{center}
$Row^{\lambda / \mu}:=\{\sigma \in \mathfrak{S}_d \mid \sigma\textrm{ permutes the enties in each row among themselves}\}$
\end{center}
and
\begin{center}
$Col^{\lambda / \mu}:=\{\sigma \in \mathfrak{S}_d \mid \sigma\textrm{ permutes the entries in each column among themselves}\}$.
\end{center}

\noindent In the group algebra $\mathbb{C}[\mathfrak{S}_d]$ we introduce two elements, called the Young symmetrizers
\begin{center}
$\Upsilon_{W}^{\lambda / \mu}:=\displaystyle \sum_{\substack{\sigma \in Row^{\lambda / \mu} \\ \rho \in Col^{\lambda / \mu}}} sign(\rho)\sigma \rho$
\end{center}
and
\begin{center}
$\Upsilon_{S}^{\lambda / \mu}:=\displaystyle \sum_{\substack{\sigma \in Row^{\lambda / \mu} \\ \rho \in Col^{\lambda / \mu}}} sign(\rho)\rho \sigma$.
\end{center}

\noindent Let $V=\mathbb{C}^N$ with standard basis $\{e_1,\ldots,e_N\}$. The symmetric group $\mathfrak{S}_{d}$ acts on the $d$th tensor product $V^{\otimes d}$ on the right by permuting the factors, while $\GL_N$ acts on $V$ on the left and thus diagonally on $V^{\otimes d}$ on the left. The fact that this left action of $\GL_N$ commutes with the right action of $\mathfrak{S}_d$ is the source of Schur-Weyl duality and gives the relationship between the irreducible finite-dimensional representations of the general linear and symmetric groups.

The Schur Module $\mathbb{S}^{\lambda / \mu}(V)$ and Weyl Module $\mathbb{W}^{\lambda / \mu}(V)$ are defined to be
\begin{center}
$\mathbb{S}^{\lambda / \mu}(V):=(V^{\otimes d})\Upsilon_{S}^{\lambda / \mu}$
\end{center}
and
\begin{center}
$\mathbb{W}^{\lambda / \mu}(V):=(V^{\otimes d})\Upsilon_{W}^{\lambda / \mu}$.
\end{center}

These are $\GL_N$ representations spanned by all the young symmetrized tensors in $V^{\otimes d}$. In characteristic zero these representations are related by the identity $\mathbb{S}^{\lambda / \mu}(V)\cong \mathbb{W}^{\lambda' / \mu'}(V)$\cite[Prop. 2.1.18(c)]{MR1988690} and for our purposes it will prove to be convenient to focus on the Weyl Modules $\mathbb{W}^{\lambda / \mu}(V)$.

Given a tableaux $\tab$ of $\lambda / \mu$ numbered with $\{ 1,\ldots,N \}$ we can associate to $\tab$ a decomposable tensor 
\begin{center}
$e_{\tab}=\displaystyle \bigotimes_{i=1}^{\lambda_1}e_{\tab(-,i)}$
\end{center}
where $e_{\tab(-,i)}$ is the tensor product, in order, of those basis vectors whose indices appear in column $i$. We will write $\wmodel{\tab}{\lambda / \mu}$ for the element $e_{\tab} \Upsilon_{W}^{\lambda / \mu}$ of $\mathbb{W}^{\lambda / \mu}(V)$ and will omit the superscript and simply write $\wmodel{\tab}{}$ so long as no confusion will arise. Multilinearity implies that $\mathbb{W}^{\lambda / \mu}(V)$ is spanned by $\wmodel{\tab}{}$ as $\tab$ ranges over all tableau on $\lambda / \mu$. In fact we may do better, as the following theorem illustrates.

\begin{theorem}
\label{theorem:ssytBasisWeylModule}
The set $\{\wmodel{\tab}{} \mid \tab\textrm{ is a semistandard tableaux on }\lambda / \mu \}$ is a $\mathbb{C}$-basis for $\mathbb{W}^{\lambda / \mu}(V)$.
\end{theorem}
\begin{proof}
This is a well known result, and a sketch of the details may be found in \cite[Exercise~6.15~and~6.19]{MR1153249}
\end{proof}

Note that the above construction works for any Young diagram $\lambda$ merely by noting that $\lambda$ has the same diagram as the skew diagram $\lambda / \emptypart$. In fact the Weyl Module $\mathbb{W}^{\lambda}(V):=\mathbb{W}^{\lambda / \emptypart}(V)$ is an irreducible polynomial $\GL_N$ representation and any Weyl Module $\mathbb{W}^{\lambda / \mu}(V)$ can be written as a direct sum of these irreducible modules by
\begin{equation}
\label{equation:glnweylmoduledecomp}
\mathbb{W}^{\lambda / \mu}(V)\cong \displaystyle \bigoplus_{\nu} {\mathbb{W}^{\nu}(V)}^{\oplus c_{\mu, \nu}^{\lambda}}
\end{equation}
where the $c_{\lambda \mu}^{\nu}$ are the Littlewood-Richardson coefficients defined in Section \ref{subsec:youngDiagramsSSYT}, and the direct sum is over all partitions $\nu$ such that $|\nu|=|\lambda|-|\mu|$.

\begin{definition} 
\label{definition:determinantRep}
Let $r$ be a positive integer. Define the $\GL_N$-representation $det_N^{r}:\GL_N\rightarrow \mathbb{C}^{*}$, $det_N^{r}(g)=(det(g))^{r}$, $g\in \GL_N$. Then $det_N^{-r}$ is defined to be the dual of $det_{N}^{r}$. 
\end{definition}

Recall the following result due originally to Schur.
\begin{theorem}
\label{theorem:greenPolyCharIso}
Let $V$ and $W$ be two finite dimensional polynomial $\GL_N$-representations. Then $V$ and $W$ are isomorphic if and only if $char(V)=char(W)$.
\end{theorem}
\begin{proof}
For a proof of this refer to \cite[Theorem 3.5 and the second remark following the proof]{MR2349209}.
\end{proof}
\begin{corollary}
\label{corollary:ratCharIso}
Let $V$ and $W$ be two finite dimensional rational $\GL_N$-representations. Then $V$ and $W$ are isomorphic if and only if $char(V)=char(W)$.
\end{corollary}
\begin{proof}
Tensoring $V$ and $W$ by a sufficient power of the $\GL_N$ determinant representation reduces this to the polynomial $\GL_N$-representation case.
\end{proof}

We have the following isomorphism of $\GL_N$-modules by \eqref{equation:skewSchurPi} and Theorem \ref{theorem:greenPolyCharIso}
\begin{equation}
\label{equation:dualofweylmodule}
\mathbb{W}^{(m^n)/\lambda}(\mathbb{C}^N) \cong \mathbb{W}^{((m^n)/\lambda)^{\pi}}(\mathbb{C}^N)
\end{equation}
for $m$, $n$ two positive integers with $m < N$. By Remark \ref{remark:pirotpartition} these are irreducible $\GL_N$-modules.

\section{Decomposition results}
\label{sec:Main}

The ultimate goal of this section is a combinatorial description of the decomposition of the homogeneous coordinate ring of a Schubert variety into irreducible modules for the action of the Levi subgroups of certain parabolic subgroups of $\GL_N$.

\subsection{Blocks, heads, and partitions in degree 1}
\label{subsec:BlockHeadClassPartitionDeg1}
\numberwithin{equation}{subsection}

Let $P=P_{d}$ and $w:=(\ell_1,\ldots,\ell_{d})\in W^{P}$. Then $X(w)$ is a Schubert variety in $\Grass_{d,N}$. Let $Q_w$ be the stabilizer of $X(w)$ in $\GL_N$ for the action of left multiplication. Throughout this paper when we discuss the stabilizers of Schubert varieties in $\GL_N$ it will always be for the action of left multiplication. 
\begin{proposition}
\label{proposition:FormOfStabw}
Define 
\begin{center}
$\widehat{R}_{Q_w} \coloneqq \Set{n\in\left\{1,\ldots,N-1\right\}\mid \exists m\text{ with } n=\ell_m\text{ and }\ell_m + 1\neq \ell_{m+1}}$.
\end{center}
Then $Q_w=P_{R_{Q_w}}$ where $R_{Q_w}=\left\{1,\ldots,N-1\right\}\setminus \widehat{R}_{Q_w}$.
\end{proposition}
\begin{proof}
This is immediate from the fact that $R_{Q_w}=\{ m \in \{1,\ldots,N-1\} \mid s_{\alpha_m} w \leq w \}$\cite{MR0354698}.
\end{proof}

Let $Q$ be a parabolic subgroup of $\GL_N$ that is a subgroup of $Q_w$, then we have $Q=P_{R_Q}$ for some $R_Q\subseteq R_{Q_w}$(cf. Remark \ref{remark:formofParabolics}). Our main group of interest will be the reductive group $L$, defined as the Levi part of $Q$, and its Lie algebra $\mathfrak{l}=Lie(L)$. The group $L$ acts on $X(w)$ by left multiplication and this induces an action on the coordinate ring $\mathbb{C}[X(w)]$. This in turn induces an action of $\mathfrak{l}:=Lie(L)$ on $\mathbb{C}[X(w)]$. We explore this induced action in greater depth in Section \ref{subsec:PartialOrderLieLaction}. 

\begin{notation} 
\label{notation:nBlocksPart}
We establish some notation that will be used for the rest of the paper. Define $H_w\coloneqq \left\{\tau \in W^{P}\mid\tau \leq w\right\}$.  Let $\widehat{R}_Q:=\{1,\ldots,N-1\}\setminus R_Q$, and then let $\dprl{L}:=|\widehat{R}_Q|+1$. The set $\widehat{R}_Q$ can be written uniquely as the ascending sequence $(a_1,\ldots,a_{\dprl{L}-1})$. Set $a_{0}=0$ and $a_{\dprl{L}}=N$.
\end{notation}

\begin{definition}
\label{definition:Block} 
We may partition $\{1,\ldots,N\}$ into subsets $\Bl_{L,k} \coloneqq \left\{a_{k-1}+1,\ldots,a_{k}\right\}$ for $1\leq k\leq \dprl{L}$. We will refer to these as the \emph{blocks} of $L$. Let $N_k:=|\Bl_{L,k}|=a_{k}-a_{k-1}$. Thus $\dprl{L}$ is the number of blocks of $L$.
\end{definition}

\begin{remark}
\label{remark:formofL}
These blocks are closely related to the form of $L$ and $\mathfrak{l}$. In particular, $L=GL_{N_1}\times \cdots \times GL_{N_{\dprl{L}}}$ and $\mathfrak{l}=\mathfrak{gl}_{N_1}\times \cdots \times \mathfrak{gl}_{N_{\dprl{L}}}$. Thus our decomposition of the $L$-module $\mathbb{C}[X(w)]$ into irreducible $L$-modules will be in terms of tensor products of Weyl modules associated to the $\GL_{N_{i}}$.
\end{remark}

\begin{definition}
\label{definition:Class} 
Given an element $\tau=(i_1,\ldots,i_d)\in W^{P}$, we define the \emph{class} of $\tau$ as the sequence $\Cl_{\tau}:=(u_1,\ldots,u_{d})$ where each $u_n$ is equal to the unique $k$ such that $i_n \in \Bl_{L,k}$.
\end{definition}

\begin{proposition}
\label{proposition:HeadTypeW} 
Let $\tau \in H_w$. The following properties of $\tau$ are equivalent.
\begin{enumerate}[label=(\roman*)]
\item The subvariety $X(\tau)$ is $L$-stable.
\item The $wt(\pl_{\tau})$ is weakly $L$-dominant.
\item For all $1\leq k\leq \dprl{L}$ we have $\tau\bigcap \Bl_{L,k}$ is either empty or is a maximal collection of elements from $\Bl_{L,k}$; explicitly for all $m \in \tau\bigcap \Bl_{L,k}$ and all $n \in \Bl_{L,k} \setminus \left\{\tau\bigcap \Bl_{L,k}\right\}$ we have $m>n$.
\end{enumerate}
\end{proposition}
\begin{proof}
Let $Q_{\tau}$ be the stabilizer of $X(\tau)$.  The $L$-stability of $X(\tau)$ is equivalent to $Q \subseteq Q_{\tau}$, or that $R_Q \subseteq R_{Q_{\tau}}$. 

\noindent $\textit{(iii)} \Rightarrow \textit{(i)}:$ Suppose $\textit{(iii)}$ holds for $\tau$. We have $Q_{\tau}=P_{R_{Q_{\tau}}}$ where 
\begin{center}
$R_{Q_{\tau}}=\{ m \in \left\{1,\ldots,N-1\right\} \mid s_{\alpha_m} \tau \leq \tau \}$.
\end{center}
Suppose $m \in \widehat{R}_{Q_{\tau}}$, then $s_{\alpha_m} \tau > \tau$. The description of $\tau$ given in $\textit{(iii)}$ implies that $m=a_k$ for some $1\leq k < \dprl{L}$. Thus $m \in \widehat{R}_Q$, which implies $\widehat{R}_{Q_{\tau}} \subseteq \widehat{R}_Q$. It follows that $R_Q \subseteq R_{Q_{\tau}}$.

\noindent $\textit{(i)} \Rightarrow \textit{(iii)}:$ Suppose $\textit{(iii)}$ does not hold for $\tau$. This implies there is a block, say $\Bl_{L,k}$, such that there is a $m \in \tau \bigcap \Bl_{L,k}$ with $m+1 \in \Bl_{L,k} \setminus \left\{\tau\bigcap \Bl_{L,k}\right\}$. Then $m \in \tau$ and $m+1 \notin \tau$. This implies that $s_{\alpha_m} \tau > \tau$ and thus $m \notin  R_{Q_{\tau}}$.  Since $m$ and $m+1$ are both in $\Bl_{L,k}$ we have that $m$ is not the maximal element in the block and thus $m \in R_Q$. Thus $R_Q \not\subseteq R_{Q_{\tau}}$ and hence $Q \not\subseteq Q_{\tau}$.

\noindent $\textit{(iii)} \Rightarrow \textit{(ii)}:$ As discussed in Remark \ref{remark:weightPlucker} we have that the weight of $\pl_{\tau}$ is given in the $\epsilon $-basis by the sequence $\chi_{\tau}:=(\chi_1,\ldots,\chi_{N})$ where
\begin{center}
$\chi_i \coloneqq \left\{
\begin{array}{ll}
      -1 & i\in \tau\\
      0 & i\notin \tau 
\end{array} 
\right.$
\end{center}
for all $1\leq i \leq N$.

Then by Remark \ref{remark:formofL} and Remark \ref{remark:dominantWeightProduct} we have that  $\chi_{\tau}$ is weakly $L$-dominant if and only if when we partition the sequence $(\chi_1,\ldots,\chi_{N})$ into subsequences $\chi_{\tau}^{(k)} \coloneqq (\chi_{a_{k-1}+1},\ldots,\chi_{a_{k}})$ for $1\leq k\leq \dprl{L}$, each sequence $\chi_{\tau}^{(k)}$ is non-increasing. Since $\textit{(iii)}$ holds for $\tau$ we have that each $\chi_{\tau}^{(k)}$ is of the form $(0,\ldots,0,-1,\ldots,-1)$ which is non-increasing. Thus $\chi_{\tau}$ is weakly $L$-dominant.

\noindent $\textit{(ii)} \Rightarrow \textit{(iii)}:$ Suppose $\textit{(iii)}$ does not hold for $\tau$. Then there is a block, say $\Bl_{L,k}$, such that there is a $m \in \tau\bigcap \Bl_{L,k}$ with $m+1 \in \Bl_{L,k} \setminus \left\{\tau\bigcap \Bl_{L,k}\right\}$. But then $\chi_{\tau}^{(k)}$ is not non-increasing since $m,m+1 \in \left\{a_{k-1}+1,\ldots,a_{k}\right\}$ and $\chi_{m}=-1$ and $\chi_{m+1}=0$. Hence $\chi_{\tau}$ is not weakly $L$-dominant.
\end{proof}
\begin{definition}
\label{definition:HeadTypeW}
Let $\tau \in H_w$. If any of the three equivalent properties from Proposition \ref{proposition:HeadTypeW} hold for $\tau$ we call $\tau$ a \emph{head of type} $L$. And we define 
\begin{center}
$\He_L \coloneqq \left\{\tau \in H_w\mid \tau \text{ is a head of type }L \right\}$.
\end{center}
\end{definition}

\begin{example}
\label{example:blocks}
Let $d=3$ and $N=9$. Consider $w=(3,6,9)\in W^{P_{3}}$. Then $X(w)$ is a Schubert variety in $\Grass_{3,9}$. In this case $\widehat{R}_{Q_w} = \{3,6\}$ and $R_{Q_w}=\{1,2,4,5,7,8\}$. Choose $R_Q=R_{Q_w}$ for the parabolic subgroup $Q=P_{R_Q}$. Then $(a_0,a_1,a_2,a_3)=\{0,3,6,9\}$. So $\Bl_{L,1}=(1,2,3)$, $\Bl_{L,2}=(4,5,6)$, and $\Bl_{L,3}=(7,8,9)$. Then
\begin{center} 
$\He_{L}=\{(1,2,3),(2,3,6),(2,3,9),(3,5,6),(3,6,9)\}$.\end{center}
The head $(2,3,6)$ has $\Cl_{(2,3,6)}=(1,1,2)$.
\end{example}

We now prove a handful of technical lemmas relating to heads of type $L$, blocks of $L$, and classes that will prove useful throughout this section. Our first goal will be to describe a particular partition, that depends on $L$, of the Hasse diagram for $H_w$ into disjoint subdiagrams. This partition will turn out to influence the decomposition of $\mathbb{C}[X(w)]$ into irreducible $L$-modules. 

\begin{lemma}
\label{lemma:classGeqHeadGeq}
Let $\theta_1:=(x_1,\ldots,x_d), \theta_2:=(y_1,\ldots,y_d) \in \He_L$. Let $\Cl_{\theta_1}:=(t_1,\ldots,t_d)$ and $\Cl_{\theta_2}:=(u_1,\ldots,u_d)$. If $t_j \geq u_j$ for all $1 \leq j \leq d$, then $\theta_1 \geq \theta_2$.
\end{lemma}
\begin{proof}
Suppose that $\theta_1 \not\geq \theta_2$. This implies there must be an index $n$ such that $y_{n}>x_{n}$. Our hypothesis implies $t_{n} \geq u_{n}$, however, if $t_{n}> u_{n}$ then $x_{n}>y_{n}$, hence $t_{n}= u_{n}$.  

Let $p$ be the maximum integer such that $t_{p}= t_{n}$, and $q$ the maximum integer such that $u_{q}= u_{n}$. We know $q \geq p$, otherwise $p > q$ would imply $u_{p} > t_{p}$, which is a contradiction of our hypothesis.

But then $x_{n},\ldots,x_{p}$ and $y_{n},\ldots,y_{q}$ are both maximal sequences in $\Bl_{L,t_{n}}$ by Proposition \ref{proposition:HeadTypeW}(iii). However, the length of the sequence $y_{n},\ldots,y_{q}$ is longer or equal to $x_{n},\ldots,x_{p}$, so $x_{n}\geq y_{n}$. This is a contradiction as $n$ was chosen to be the index where $y_{n}>x_{n}$. Thus it must be the case that $\theta_1 \geq \theta_2$.
\end{proof}

\begin{lemma}
\label{lemma:HeadClassUnique}
Let $\theta_1$, $\theta_2 \in \He_L$. Then $\theta_1 = \theta_2$ if and only if $\Cl_{\theta_1} = \Cl_{\theta_2}$.
\end{lemma}
\begin{proof}
$(\Leftarrow )$ Suppose $\Cl_{\theta_1} = \Cl_{\theta_2}$. This would imply that $| \theta_1 \bigcap \Bl_{L,k} |=| \theta_2 \bigcap \Bl_{L,k} |$ for all $1\leq k\leq \dprl{L}$. Since each of these is a head, $\theta_i \bigcap \Bl_{L,k}$ for $i=1,2$ must be a maximal collection of elements from $\Bl_{L,k}$. These two results imply that $\theta_1 \bigcap \Bl_{L,k} = \theta_2 \bigcap \Bl_{L,k}$ for all $1\leq k\leq \dprl{L}$. But this, combined with the fact that the blocks of type $L$ partition $\left\{1,\ldots,N\right\}$, implies that $\theta_1 = \theta_2$. 

\noindent $(\Rightarrow )$ Suppose $\Cl_{\theta_1} \neq \Cl_{\theta_2}$. Then for some block, say $\Bl_{L,k}$, we have $| \theta_1 \bigcap \Bl_{L,k} | \neq | \theta_2 \bigcap \Bl_{L,k} |$. But this implies $\theta_1 \neq \theta_2$.
\end{proof}
\begin{lemma}
\label{lemma:ClassUnchangedR_Q}
Let $m\in R_Q$ and $\tau=(i_1,\ldots,i_d) \in H_w$. Then $\Cl_{\tau}=\Cl_{s_{\alpha_m}\tau}$.
\end{lemma}
\begin{proof}
We have $m \in \Bl_{L,k}$ for some $1 \leq k \leq \dprl{L}$. Suppose $m+1\notin \Bl_{L,k}$. This would mean that $m$ is the maximal element in $\Bl_{L,k}$, so $m=a_{k}$. Thus $m\in \widehat{R}_Q$ or $m=N$. In either case this means $m\notin R_Q$. This is a contradiction and thus $m+1\in \Bl_{L,k}$.
Now $s_{\alpha_m}$ acts on $\tau$ in one of the following ways.

\textbf{Case 1:} $\exists n$ such that $i_n = m$, with $i_{n+1} \neq m+1$. Then $s_{\alpha_m}\tau = (i_1,\ldots,i_{n-1},m+1,i_{n+1},\ldots,i_d)$. 

\textbf{Case 2:} $\exists n$ such that $i_n = m+1$, with $i_{n-1} \neq m$. Then $s_{\alpha_m}\tau = (i_1,\ldots,i_{n-1},m,i_{n+1},\ldots,i_d)$.

\textbf{Case 3:} $\exists n$ such that $i_n = m$, with $i_{n+1} = m+1$. Then $s_{\alpha_m}\tau = \tau$.

\textbf{Case 4:} $\nexists n$ such that $i_n = m$ or $i_n = m+1$. Then $s_{\alpha_m}\tau = \tau$.

Thus in all four possible cases it can be seen that $\Cl_{\tau}=\Cl_{s_{\alpha_m}\tau}$.
\end{proof}

\begin{lemma}
\label{lemma:EachElementHasHead}
Let $\tau:=(i_1,\ldots,i_d) \in H_w$. Then $\tau = s_{\alpha_{m_1}} \cdots s_{\alpha_{m_t}}\theta$ for some $m_1,\ldots,m_t \in R_Q$ and $\theta \in \He_L$.
\end{lemma}
\begin{proof}
To see this fix a $k$ with $1\leq k \leq \dprl{L}$ and consider $\Bl_{L,k}$. If $\tau \bigcap \Bl_{L,k}$ is empty we are done with this block. Otherwise $\tau \bigcap \Bl_{L,k} = (i_{n},i_{n+1},\ldots,i_{n+z})$ for some $1\leq n,z \leq d $ with $n+z\leq d$. If this is a maximal collection of elements in $\Bl_{L,k}$ we are done with this block, otherwise we may make it maximal. 

As we saw in Lemma \ref{lemma:ClassUnchangedR_Q}, case 1, we may act on $\tau$ by certain $s_{\alpha_m}$ to increment $(i_{n},\ldots,i_{n+z})$ to $(a_{k}-z,\ldots,a_k)$ (cf. Notation \ref{notation:nBlocksPart}). These $s_{\alpha_m}$ will all have $m \in R_Q$, since otherwise, if $m\in \widehat{R}_Q$, then acting by $s_{\alpha_m}$ would increment an entry to a value in $\Bl_{L,k+1}$. Further, all these $s_{\alpha_m}$ will have $m \in \Bl_{L,k}$. Thus these actions only affect the entries of $\tau$ that intersect with $\Bl_{L,k}$ and so can be performed independently for each block.

After performing the incrementing process for each block we have $s_{\alpha_{m_t}} \cdots s_{\alpha_{m_1}} \tau$ is a head of type $L$ for some $m_1,\ldots,m_t \in R_Q$. That is $s_{\alpha_{m_t}} \cdots s_{\alpha_{m_1}} \tau = \theta$ for some $\theta \in \He_L$. Rewriting we get our desired result $\tau = s_{\alpha_{m_1}} \cdots s_{\alpha_{m_t}}\theta$.
\end{proof}

Let $\mathcal{H}_w$ be the Hasse diagram for the Bruhat order on $H_w$. We may label the edges of $\mathcal{H}_w$ in the following way. Given an edge connecting $\tau_1$ to $\tau_2$ with $\tau_1\leq \tau_2$ we know that $\tau_1 = s_{\beta}\tau_2$ for a unique $\beta \in \Phi^+$. However, in the case of the Grassmannian, we know that $\beta$ is a simple root. This is because a divisor of $X(\tau)$, for $\tau=(i_1,\ldots,i_d)$, is obtained by reducing a single entry, say $i_n$, to $i_n - 1$, in which case $\beta$ is simply $\alpha_{i_n - 1}$. Thus we may label the edge connecting $\tau_1$ to $\tau_2$ by the unique $s_{\alpha_m}$ such that $\tau_1 = s_{\alpha_m}\tau_2$, $\alpha_m \in S$.

\begin{proposition}
\label{proposition:PartitionofHasse}
(Partition in Degree 1) Let $\widehat{\mathcal{H}}_w$ be the diagram formed by removing all edges of $\mathcal{H}_w$ labeled by  $s_{\alpha_m}$ with $m\in \widehat{R}_Q$. Then $\widehat{\mathcal{H}}_w$ is a disconnected diagram with $|\He_L|$ disjoint subdiagrams and each subdiagram has a unique maximal element under the Bruhat order given by a $\theta \in \He_L$. Further the class of each element in a fixed subdiagram is equal to the class of the head in that subdiagram.
\end{proposition}
\begin{proof}
By Lemma \ref{lemma:EachElementHasHead} we have that for any $\tau \in H_w$, $\tau = s_{\alpha_{m_1}} \cdots s_{\alpha_{m_t}}\theta$ for some $m_1,\ldots,m_t \in R_Q$ and $\theta \in \He_L$. Thus there is a path of edges in $\widehat{\mathcal{H}}_w$ connecting $\tau$ and $\theta$. By Lemma \ref{lemma:ClassUnchangedR_Q} this also means $\Cl_{\tau} = \Cl_{\theta}$. Combining this with Lemma \ref{lemma:HeadClassUnique} we get that $\theta$ is in fact the unique head connected to $\tau$ in $\widehat{\mathcal{H}}_w$. And thus removing all edges of $\mathcal{H}_w$ labeled by  $s_{\alpha_m}$ with $m\in \widehat{R}_Q$ to form $\widehat{\mathcal{H}}_w$ we get a disconnected diagram with $|\He_L|$ disjoint subdiagrams.

It remains to show that the unique maximal element in each subdiagram is in fact the head. But this is clear by the proof of Lemma \ref{lemma:EachElementHasHead}. For every $\tau \in H_w$ we found $s_{\alpha_{m_t}} \cdots s_{\alpha_{m_1}} \tau = \theta$ for some $m_1,\ldots,m_t \in R_Q$ and $\theta \in \He_L$. Each subsequent action by the $s_{\alpha_{m_n}}$ resulted in an element that was greater in the Bruhat order, and thus for all $\tau$ connected to $\theta$ in $\widehat{\mathcal{H}}_w$ we have $\theta\geq \tau$.  
\end{proof}
\begin{definition}
\label{definition:HeadofTau} 
Let $\tau \in H_w$. Then define the \emph{head of $\tau$}, which we will denote $\theta_{\tau}$, to be the unique head in $\He_L$ connected to $\tau$ in $\widehat{\mathcal{H}}_w$. This is well defined and unique by Proposition \ref{proposition:PartitionofHasse}.

\end{definition}

\begin{example}
\label{example:degree1Partition}
In Example \ref{example:blocks} we saw that for $w=(3,6,9)$ and $Q=P_{R_Q}$ with $R_Q=R_{Q_w}=\left\{ 1,2,4,5,7,8 \right\}$ we had $\He_{L}=\{(1,2,3),(2,3,6),(2,3,9),(3,5,6),(3,6,9)\}$. If we draw the Hasse diagram for $w$ and remove the edges labeled by $s_{\alpha_m}$ for $m \in \widehat{R}_{Q}=\left\{ 3,6 \right\}$ we have the following disjoint diagram. The diagram has $\mid \He_{L} \mid = 5$ disjoint subdiagrams, with the unique maximal element in each subdiagram a head of type $L$.

\end{example}
\begin{tikzpicture}[scale=.7]
  \def\tikzhassevert{1.3}
  \def\tikzhassehor{2.75}
  \def\tikzhassevertoffset{0}
  \node (3-6-9) at (0,4*\tikzhassevert + \tikzhassevertoffset) {$\textbf{(3,6,9)}$};
  
  \node (2-6-9) at (-1*\tikzhassehor,3*\tikzhassevert + \tikzhassevertoffset) {$(2,6,9)$};
  \node (3-5-9) at (0,3*\tikzhassevert + \tikzhassevertoffset) {$(3,5,9)$};
  \node (3-6-8) at (1*\tikzhassehor,3*\tikzhassevert + \tikzhassevertoffset) {$(3,6,8)$};
  
  \node (1-6-9) at (-2.5*\tikzhassehor,2*\tikzhassevert + \tikzhassevertoffset) {$(1,6,9)$};
  \node (2-5-9) at (-1.5*\tikzhassehor,2*\tikzhassevert + \tikzhassevertoffset) {$(2,5,9)$};
  \node (2-6-8) at (-0.5*\tikzhassehor,2*\tikzhassevert + \tikzhassevertoffset) {$(2,6,8)$};
  \node (3-4-9) at (0.5*\tikzhassehor,2*\tikzhassevert + \tikzhassevertoffset) {$(3,4,9)$};
  \node (3-5-8) at (1.5*\tikzhassehor,2*\tikzhassevert + \tikzhassevertoffset) {$(3,5,8)$};
  \node (3-6-7) at (2.5*\tikzhassehor,2*\tikzhassevert + \tikzhassevertoffset) {$(3,6,7)$};
  
  \node (1-5-9) at (-3*\tikzhassehor,1*\tikzhassevert + \tikzhassevertoffset) {$(1,5,9)$};
  \node (1-6-8) at (-2*\tikzhassehor,1*\tikzhassevert + \tikzhassevertoffset) {$(1,6,8)$};
  \node (2-4-9) at (-1*\tikzhassehor,1*\tikzhassevert + \tikzhassevertoffset) {$(2,4,9)$};
  \node (2-5-8) at (0,1*\tikzhassevert + \tikzhassevertoffset) {$(2,5,8)$};
  \node (2-6-7) at (1*\tikzhassehor,1*\tikzhassevert + \tikzhassevertoffset) {$(2,6,7)$};
  \node (3-4-8) at (2*\tikzhassehor,1*\tikzhassevert + \tikzhassevertoffset) {$(3,4,8)$};
  \node (3-5-7) at (3*\tikzhassehor,1*\tikzhassevert + \tikzhassevertoffset) {$(3,5,7)$};
  
  \node (2-3-9) at (-3.5*\tikzhassehor - 0.5,0) {$\textbf{(2,3,9)}$};
  \node (1-4-9) at (-2.5*\tikzhassehor,0 + \tikzhassevertoffset) {$(1,4,9)$};
  \node (1-5-8) at (-1.5*\tikzhassehor,0 + \tikzhassevertoffset) {$(1,5,8)$};
  \node (1-6-7) at (-0.5*\tikzhassehor,0 + \tikzhassevertoffset) {$(1,6,7)$};
  \node (2-4-8) at (0.5*\tikzhassehor,0 + \tikzhassevertoffset) {$(2,4,8)$};
  \node (2-5-7) at (1.5*\tikzhassehor,0 + \tikzhassevertoffset) {$(2,5,7)$};
  \node (3-4-7) at (2.5*\tikzhassehor,0 + \tikzhassevertoffset) {$(3,4,7)$};
  \node (3-5-6) at (3.5*\tikzhassehor+0.5,0) {$\textbf{(3,5,6)}$};
  
  \node (1-3-9) at (-3*\tikzhassehor - 0.75,-1*\tikzhassevert) {$(1,3,9)$};
  \node (2-3-8) at (-2*\tikzhassehor - 0.75,-1*\tikzhassevert) {$(2,3,8)$};  
  \node (1-4-8) at (-1*\tikzhassehor,-1*\tikzhassevert + \tikzhassevertoffset) {$(1,4,8)$};
  \node (1-5-7) at (0,-1*\tikzhassevert + \tikzhassevertoffset) {$(1,5,7)$};
  \node (2-4-7) at (1*\tikzhassehor,-1*\tikzhassevert + \tikzhassevertoffset) {$(2,4,7)$};
  \node (2-5-6) at (2*\tikzhassehor +0.75,-1*\tikzhassevert) {$(2,5,6)$};
  \node (3-4-6) at (3*\tikzhassehor + 0.75,-1*\tikzhassevert) {$(3,4,6)$};
  
  \node (1-2-9) at (-3*\tikzhassehor -0.75,-2*\tikzhassevert) {$(1,2,9)$};
  \node (1-3-8) at (-2*\tikzhassehor-0.75,-2*\tikzhassevert) {$(1,3,8)$};
  \node (2-3-7) at (-1*\tikzhassehor-0.75,-2*\tikzhassevert) {$(2,3,7)$};
  \node (1-4-7) at (0,-2*\tikzhassevert + \tikzhassevertoffset) {$(1,4,7)$};
  \node (1-5-6) at (1*\tikzhassehor+0.75,-2*\tikzhassevert) {$(1,5,6)$};
  \node (2-4-6) at (2*\tikzhassehor+0.75,-2*\tikzhassevert) {$(2,4,6)$};
  \node (3-4-5) at (3*\tikzhassehor+0.75,-2*\tikzhassevert) {$(3,4,5)$};
  
  \node (1-2-8) at (-2*\tikzhassehor,-3*\tikzhassevert) {$(1,2,8)$};
  \node (1-3-7) at (-1*\tikzhassehor,-3*\tikzhassevert) {$(1,3,7)$};
  \node (2-3-6) at (0,-3*\tikzhassevert) {$\textbf{(2,3,6)}$};
  \node (1-4-6) at (1*\tikzhassehor,-3*\tikzhassevert) {$(1,4,6)$};
  \node (2-4-5) at (2*\tikzhassehor,-3*\tikzhassevert) {$(2,4,5)$};
  
  \node (1-2-7) at (-1.5*\tikzhassehor,-4*\tikzhassevert) {$(1,2,7)$};
  \node (1-3-6) at (-0.5*\tikzhassehor,-4*\tikzhassevert) {$(1,3,6)$};
  \node (2-3-5) at (0.5*\tikzhassehor,-4*\tikzhassevert) {$(2,3,5)$};
  \node (1-4-5) at (1.5*\tikzhassehor,-4*\tikzhassevert) {$(1,4,5)$};
  
  \node (1-2-6) at (-1*\tikzhassehor,-5*\tikzhassevert) {$(1,2,6)$};
  \node (1-3-5) at (0,-5*\tikzhassevert) {$(1,3,5)$};
  \node (2-3-4) at (1*\tikzhassehor,-5*\tikzhassevert) {$(2,3,4)$};
  
  \node (1-2-5) at (-0.5*\tikzhassehor,-6*\tikzhassevert) {$(1,2,5)$};
  \node (1-3-4) at (0.5*\tikzhassehor,-6*\tikzhassevert) {$(1,3,4)$};
  
  \node (1-2-4) at (0,-7*\tikzhassevert) {$(1,2,4)$};
  
  \node (1-2-3) at (0,-8*\tikzhassevert) {$\textbf{(1,2,3)}$};
  
  \draw (3-6-9) -- (2-6-9);
  \draw (3-6-9) -- (3-5-9);
  \draw (3-6-9) -- (3-6-8);
  
  \draw (2-6-9) -- (1-6-9);
  \draw (2-6-9) -- (2-5-9);
  \draw (2-6-9) -- (2-6-8);
  \draw (3-5-9) -- (2-5-9);
  \draw (3-5-9) -- (3-4-9);
  \draw (3-5-9) -- (3-5-8);
  \draw (3-6-8) -- (2-6-8);
  \draw (3-6-8) -- (3-5-8);
  \draw (3-6-8) -- (3-6-7);
  
  \draw (1-6-9) -- (1-5-9);
  \draw (1-6-9) -- (1-6-8);
  \draw (2-5-9) -- (1-5-9);
  \draw (2-5-9) -- (2-4-9);
  \draw (2-5-9) -- (2-5-8);
  \draw (2-6-8) -- (1-6-8);
  \draw (2-6-8) -- (2-5-8);
  \draw (2-6-8) -- (2-6-7);
  \draw (3-4-9) -- (2-4-9);
  \draw (3-4-9) -- (3-4-8);
  \draw (3-5-8) -- (2-5-8);
  \draw (3-5-8) -- (3-4-8);
  \draw (3-5-8) -- (3-5-7);
  \draw (3-6-7) -- (2-6-7);
  \draw (3-6-7) -- (3-5-7);
  
  \draw (1-5-9) -- (1-4-9);
  \draw (1-5-9) -- (1-5-8);
  \draw (1-6-8) -- (1-5-8);
  \draw (1-6-8) -- (1-6-7);
  \draw (2-4-9) -- (1-4-9);
  \draw (2-4-9) -- (2-4-8);
  \draw (2-5-8) -- (1-5-8);
  \draw (2-5-8) -- (2-4-8);
  \draw (2-5-8) -- (2-5-7);
  \draw (2-6-7) -- (1-6-7);
  \draw (2-6-7) -- (2-5-7);
  \draw (3-4-8) -- (2-4-8);
  \draw (3-4-8) -- (3-4-7);
  \draw (3-5-7) -- (2-5-7);
  \draw (3-5-7) -- (3-4-7);
  
  \draw (1-4-9) -- (1-4-8);
  \draw (1-5-8) -- (1-4-8);
  \draw (1-5-8) -- (1-5-7);
  \draw (1-6-7) -- (1-5-7);
  \draw (2-3-9) -- (1-3-9);
  \draw (2-3-9) -- (2-3-8);
  \draw (2-4-8) -- (1-4-8);
  \draw (2-4-8) -- (2-4-7);
  \draw (2-5-7) -- (1-5-7);
  \draw (2-5-7) -- (2-4-7);
  \draw (3-4-7) -- (2-4-7);
  \draw (3-5-6) -- (2-5-6);
  \draw (3-5-6) -- (3-4-6);
  
  \draw (1-3-9) -- (1-2-9);
  \draw (1-3-9) -- (1-3-8);
  \draw (1-4-8) -- (1-4-7);
  \draw (1-5-7) -- (1-4-7);
  \draw (2-3-8) -- (1-3-8);
  \draw (2-3-8) -- (2-3-7);
  \draw (2-4-7) -- (1-4-7);
  \draw (2-5-6) -- (1-5-6);
  \draw (2-5-6) -- (2-4-6);
  \draw (3-4-6) -- (2-4-6);
  \draw (3-4-6) -- (3-4-5);
    
  \draw (1-2-9) -- (1-2-8);
  \draw (1-3-8) -- (1-2-8);
  \draw (1-3-8) -- (1-3-7);
  \draw (2-3-7) -- (1-3-7);
  \draw (1-5-6) -- (1-4-6);
  \draw (2-4-6) -- (1-4-6);
  \draw (2-4-6) -- (2-4-5);
  \draw (3-4-5) -- (2-4-5);
  
  \draw (1-2-8) -- (1-2-7);
  \draw (1-3-7) -- (1-2-7);
  \draw (1-4-6) -- (1-4-5);
  \draw (2-3-6) -- (2-3-5);
  \draw (2-3-6) -- (1-3-6);
  \draw (2-4-5) -- (1-4-5);
  
  \draw (1-3-6) -- (1-2-6);
  \draw (1-3-6) -- (1-3-5);
  \draw (2-3-5) -- (1-3-5);
  \draw (2-3-5) -- (2-3-4);
  
  \draw (1-2-6) -- (1-2-5);
  \draw (1-3-5) -- (1-2-5);
  \draw (1-3-5) -- (1-3-4);
  \draw (2-3-4) -- (1-3-4);
  
  \draw (1-2-5) -- (1-2-4);
  \draw (1-3-4) -- (1-2-4);
  
\end{tikzpicture}

\begin{remark} 
It is not difficult to check that each of the subdiagrams in $\widehat{\mathcal{H}}_w$ will always be an interval in $H_w$.
\end{remark}

\begin{corollary}
\label{corollary:classDeterminesHeadEquality}
Let $\tau_1, \tau_2 \in H_w$. Then $\theta_{\tau_1}=\theta_{\tau_2}$ if and only if  $\Cl_{\tau_1}=\Cl_{\tau_2}$.
\end{corollary}
\begin{proof}
By Lemma \ref{lemma:HeadClassUnique} we have $\theta_{\tau_1}=\theta_{\tau_2}$ if and only if $\Cl_{\theta_{\tau_1}}=\Cl_{\theta_{\tau_2}}$. But then as seen in the proof of Proposition \ref{proposition:PartitionofHasse} we have $\Cl_{\tau_1}=\Cl_{\theta_{\tau_1}}$ and $\Cl_{\tau_2}=\Cl_{\theta_{\tau_2}}$. Thus $\Cl_{\theta_{\tau_1}}=\Cl_{\theta_{\tau_2}}$ if and only if $\Cl_{\tau_1} = \Cl_{\tau_2}$.
\end{proof}

\subsection{A partial order on the set of degree r heads}
\label{subsec:DegreeRHeadsPartialOrder}
We may extend the definition of a head in the following way. Let $\theta_1,\ldots,\theta_r \in \He_{L}$. Then we define the sequence $(\theta_1,\ldots,\theta_r)$ to be a\emph{ degree r head}. If, in addition, $\theta_1 \geq \cdots \geq \theta_r$, then $(\theta_1,\ldots,\theta_r)$ is a \emph{standard degree r head}.
\begin{definition}
\label{definition:stdr}
Define
\begin{center}
$\He_{L,r}:=\{(\theta_1,\ldots,\theta_r) \mid \theta_i \in \He_L\}$,
\end{center}
and
\begin{center}
$\He_{L,r}^{std}:=\{(\theta_1,\ldots,\theta_r)\in \He_{L,r} \mid \theta_1 \geq \cdots \geq \theta_r \}$,
\end{center}
the set of all degree r heads and all standard degree r heads, respectively.
\label{definition:HeadofTauDegR} 

Finally, for $\tau_1,\ldots,\tau_r \in H_w$, define the \emph{degree $r$ head of $(\tau_1,\ldots,\tau_r)$} to be $(\theta_{\tau_1},\ldots,\theta_{\tau_r})$. This degree $r$ head is clearly unique since each individual head is unique.
\end{definition}

To each $\theta \in \He_{L,1}$ $(=\He_L)$ we may associate a collection of Pl\"{u}cker coordinates $\pl_{\tau}$ such that $\tau$ has head $\theta$. This gives us a partition of the degree 1 standard monomials by Proposition \ref{proposition:PartitionofHasse}. The next step is to describe a partition of the degree r standard monomials in terms of degree r heads(cf. Corollary \ref{corollary:decompositionStdr}). The fact that this is possible is due to a remarkable property of the degree 1 heads: given two elements $\tau_1, \tau_2 \in H_w$ which satisfy $\tau_1 \geq \tau_2$, their respective degree 1 heads $\theta_{\tau_1}$, $\theta_{\tau_2}$ satisfy $\theta_{\tau_1} \geq \theta_{\tau_2}$ (as we shall see in Proposition \ref{proposition:PartitionofDegRStandard} for the case $r=1$). Note that this property does not hold for any partition of the Hasse diagram, or even any partition with each subdiagram containing a unique maximal element. 


\begin{proposition}
\label{proposition:PartitionofDegRStandard}
Let $\tau_1,\ldots,\tau_r \in H_w$, $\tau_1\geq \cdots \geq \tau_r$, with degree $r$ head $(\theta_{\tau_1},\ldots,\theta_{\tau_r})$. Then $\theta_{\tau_1}\geq \cdots\geq \theta_{\tau_r}$.
\end{proposition}
\begin{proof}
Let $i$ be an integer with $1\leq i < r$. We have that $\tau_{i}:=(x_1,\ldots,x_d) \geq \tau_{i+1}:=(y_1,\ldots,y_d)$. Let $\Cl_{\tau_i} \coloneqq (t_{1},\ldots,t_{d})$ and $\Cl_{\tau_{i+1}} \coloneqq (u_1,\ldots,u_{d})$. Suppose that $u_{j}>t_{j}$ for some $1\leq j \leq d$, then this would imply that $y_{j}>x_{j}$ which is a contradiction of $\tau_{i} \geq \tau_{i+1}$. Thus 
\begin{equation}
\label{equation:partdegr}
t_j \geq u_j \textrm{ for all }1 \leq j \leq d.
\end{equation}

By Proposition \ref{proposition:PartitionofHasse} $\Cl_{\theta_{\tau_i}}=\Cl_{\tau_i}=(t_{1},\ldots,t_{d})$ and $\Cl_{\theta_{\tau_{i+1}}}=\Cl_{\tau_{i+1}}=(u_{1},\ldots,u_{d})$.

This together with \eqref{equation:partdegr} implies, by Lemma \ref{lemma:classGeqHeadGeq}, that $\theta_{\tau_i} \geq \theta_{\tau_{i+1}}$. As our choice of $i$ was arbitrary we are done.
\end{proof}

\begin{definition}
Let 

\begin{center}
$\Std_r:=\left\{ \pl_{\tau_1} \cdots \pl_{\tau_r} \mid \tau_i \in H_w \textrm{ for } 1\leq i \leq r \textrm{ and }\tau_1 \geq \cdots \geq \tau_r \right\}$.
\end{center}
This is the set of degree r standard monomials. For $\underline{\theta}:=(\theta_1,\ldots,\theta_r)\in\He_{L,r}^{std}$ define 
\begin{center}
$\Std_{\underline{\theta}} := \left\{\pl_{\tau_1} \cdots \pl_{\tau_r} \in \Std_r \mid \tau_i \text{ has head } \theta_i \text{ for } 1\leq i \leq r \right\}$
\end{center}

We will often want to refer to the subspace of $\mathbb{C}[X(w)]_r$ generated by these subsets; for $X\subseteq Std_r$ let $\spann{X}$ denote the span of the elements in $X$.
\end{definition}
With these definitions in hand we may now state two important corollaries of Proposition \ref{proposition:PartitionofDegRStandard}.
\begin{corollary}
Let $\pl_{\tau_1} \cdots \pl_{\tau_r} \in \Std_r$. Then $\pl_{\theta_{\tau_1}} \cdots \pl_{\theta_{\tau_r}}$ is standard.
\end{corollary}

\begin{corollary}
\label{corollary:decompositionStdr}
(Partition in Degree r) The set $\Std_r$ is partitioned into disjoint subsets labeled by $\underline{\theta} \in \He_{L,r}^{std}$. 
Explicitly 
\begin{center}
$\displaystyle \Std_r=\bigsqcup_{\underline{\theta} \in \He_{L,r}^{std}} \Std_{\underline{\theta}}$.
\end{center}
And this implies
\begin{center}
$\displaystyle \spann{\Std_r}\;\;=\bigoplus_{\underline{\theta} \in \He_{L,r}^{std}} \spann{\Std_{\underline{\theta}}}$.
\end{center}
\end{corollary}
\begin{proof}
This is immediate by Proposition \ref{proposition:PartitionofDegRStandard} since each degree r standard monomial $\pl_{\tau_1} \cdots \pl_{\tau_r} \in \Std_r$ has a unique standard monomial  $\pl_{\theta_{\tau_1}} \cdots \pl_{\theta_{\tau_r}}$ such that $(\theta_{\tau_1},\ldots,\theta_{\tau_r}) \in \He_{L,r}^{std}$ is the degree r head of $(\tau_1,\ldots,\tau_r)$.
\end{proof}

When the degree is equal to 1 the $\spann{\Std_{\underline{\theta}}}$ with $\underline{\theta} \in \He_{L,1}(=\He_L)$ are $L$-stable, and in fact are irreducible $L$-modules(cf. Remark \ref{remark:decompDegree1}). Our initial hope was that this might extend to higher degrees. Unfortunately, when $r>1$, it is no longer the case that the $\spann{\Std_{\underline{\theta}}}$ are $L$-stable for all $\underline{\theta} \in \He_{L,r}^{std}$. This is due to the interaction between the $L$-action and the standard monomial straightening process. To correct for this lack of $L$-stability we introduce a partial order on the set of degree r heads, inspired by the straightening process, which will allow us to introduce new subspaces of $\spann{\Std_r}$ that are $L$-stable.

\begin{definition}
\label{definition:straighteningPartialOrder}
We now define a partial order on the set of degree $r$ heads $\He_{L,r}$ that, as we will see in Theorem \ref{theorem:degrStraightenPartialOrder}, is closely related to the straightening rule and hence shall denote it $\geq_{str}$. Define $(\theta_1,\ldots,\theta_r) >_{str} (\theta'_1,\ldots,\theta'_r)$ if there exists a $q\leq r$ such that $\theta_n = \theta'_n$ for all $n<q$ and $\theta_q > \theta'_q$ . Equality occurs if $\theta_n = \theta'_n$ for all $1\leq n \leq r$. This is just the lexicographic order on these sequences, with the order on each individual entry being the Bruhat order.
\end{definition}

\begin{definition}
\label{definition:StdTheta}
Let $\underline{\theta} \in \He_{L,r}^{std}$. Define 
\begin{center}
$\Std_{\underline{\theta}}^{\geq_{str}} = \left\{\pl_{\tau_1} \cdots \pl_{\tau_r} \in \Std_r \mid  (\theta_{\tau_1},\ldots,\theta_{\tau_r}) \geq_{str} \underline{\theta} \right\}$
\end{center}
and 
\begin{center}
$\Std_{\underline{\theta}}^{>_{str}} = \left\{\pl_{\tau_1} \cdots \pl_{\tau_r} \in \Std_r \mid  (\theta_{\tau_1},\ldots,\theta_{\tau_r}) >_{str} \underline{\theta} \right\}$.
\end{center}
\end{definition}

\begin{remark}
\label{remark:relationStdTheta}
Note that 
\begin{center}
$\Std_{\underline{\theta}} = \Std_{\underline{\theta}}^{\geq_{str}} \setminus \Std_{\underline{\theta}}^{>_{str}}$
\end{center}
which implies that we have the following isomorphism of vector spaces
\begin{center}
$\spann{\Std_{\underline{\theta}}} \cong \spann{\Std_{\underline{\theta}}^{\geq_{str}}} / \spann{\Std_{\underline{\theta}}^{>_{str}}}$.
\end{center}
\end{remark}

\noindent The goal of the next section is to show that $\spann{\Std_{\underline{\theta}}^{\geq_{str}}}$ and $\spann{\Std_{\underline{\theta}}^{>_{str}}}$ are both $L$-stable.

\subsection{The relation between \texorpdfstring{$\geq_{str}$}{the partial order} and the \texorpdfstring{$\mathfrak{l}$}{l}-action}
\label{subsec:PartialOrderLieLaction}

Let $E_{ij}$ be the $N\times N$ matrix with a $1$ in the $(i,j)$th entry, and zero in all other entries. The action of $\mathfrak{l} \subset \mathfrak{gl}_{N}$ on $\mathbb{C}[X(w)]$ is induced by the action of $\mathfrak{gl}_{N}$ on $\mathbb{C}[X(w)]$. The generators of $\mathfrak{l}$ as a Lie algebra are
\begin{center}
$X_{i}=E_{i i+1}$, $\quad X_{-i}=E_{i+1 i}$, $\quad H_{i}=E_{ii}$ for $i\in R_Q$
\end{center}
satisfying the relations
\begin{center}
$\begin{array}{rl}
[H_{i},H_{j}] & = 0 \\
\left[H_{i}, X_{\pm j}\right] & = \pm (\delta_{i j} - \delta_{i j+1}) X_{\pm j} \\
\left[X_{i}, X_{-j}\right] & = \delta_{ij}(H_{i} - H_{i+1}) \\
(adX_{\pm j})^{1-a_{ij}}X_{\pm i} & =0
\end{array}$
\end{center}
where $\delta_{ij}$ is the Kronecker delta and $a_{ij}=2\delta_{ij} - \delta_{i-1j}-\delta_{i+1j}$. Let $\tau=(i_1,\ldots,i_d) \in H_w$ and denote the integers $\{1,\ldots,N\} \setminus \{i_1,\ldots,i_d\}$ by $j_1,\ldots,j_{N-d}$(arranged in ascending order). As in Remark \ref{remark:weightPlucker} we identify the Pl\"{u}cker coordinate $\pl_{\tau} \in \mathbb{C}[X(w)]$ with the element $e_{j_1}\wedge\cdots\wedge e_{j_{N-d}} \otimes (e_1 \wedge \cdots \wedge e_N)^{*} \in \bigwedge^{N-d}\mathbb{C}^{N} \otimes (det^{1}_N)^{*}$. Using this identification and the fact that 
\begin{center}
$E_{ij} e_k  = \delta_{jk}e_{i}$
\end{center}
we may calculate the action of the algebra generators of $\mathfrak{l}$ on a Pl\"{u}cker coordinate.

Let $\pl_{\tau_1} \cdots \pl_{\tau_r} \in \Std_r$. Then for $i\in R_Q$ the action on a degree r standard monomial is given by
\begin{center}
$X_{\pm i}(\pl_{\tau_1} \cdots \pl_{\tau_r}) = \displaystyle \sum_{j=1}^{r} \pl_{\tau_1} \cdots X_{\pm i}(\pl_{\tau_j}) \cdots \pl_{\tau_r}$.
\end{center}
\begin{center}
$H_{i}(\pl_{\tau_1} \cdots \pl_{\tau_r}) = \displaystyle \sum_{j=1}^{r} \pl_{\tau_1} \cdots H_{i}(\pl_{\tau_j}) \cdots \pl_{\tau_r}$.
\end{center}

\noindent Then 
\begin{center}
\[
X_{i}(\pl_{\tau_j}) = 
\begin{cases}
        \pl_{s_{\alpha_i}\tau_j} & \textrm{if } \tau_j\textrm{ has an entry equal to }i\textrm{ and no entry equal to }i+1 \\
        0 & \text{otherwise }\\
\end{cases}
\]
\end{center}
where in the single nonzero case $s_{\alpha_i}\tau_j$ is obtained from $\tau_j$ by replacing $i$ with $i+1$. And 
\begin{center}
\[
X_{-i}(\pl_{\tau_j}) = 
\begin{cases}
        \pl_{s_{\alpha_i}\tau_j} & \textrm{if } \tau_j\textrm{ has an entry equal to }i+1\textrm{ and no entry equal to }i \\
        0 & \text{otherwise }\\
\end{cases}
\]
\end{center}
where in the single nonzero case $s_{\alpha_i}\tau_j$ is obtained from $\tau_j$ by replacing $i+1$ with $i$.

\noindent Finally
\begin{center}
\[
H_{i}(\pl_{\tau_j}) = 
\begin{cases}        
        A_{i,j} \pl_{\tau_j} & \textrm{if }\tau_j\textrm{ has no entry equal to }i \\
        0 & \textrm{otherwise} \\
\end{cases}
\]
\end{center}
where $A_{i,j}$ is some constant in $\mathbb{C}$ whose value depends on $i$ and $\tau_j$.

\begin{remark}
\label{remark:LStability}
We are primarily interested in these results for checking the $L$-stability of subspaces of $\mathbb{C}[X(w)]$, and since such a subspace is $L$-stable if and only if it is $\mathfrak{l}$-stable for the induced action, we may reduce to checking stability under the Lie algebra action. The benefit of this is that the Lie algebra action is easier to calculate. Note also that the action of the $H_i$ on a Pl\"{u}cker coordinate always results in either zero or a constant times the Pl\"{u}cker coordinate itself, and thus a subspace that has a basis of Pl\"{u}cker coordinates will always be stable under the action of the $H_i$.
\end{remark}

We would now like to investigate the interplay between the above action, the straightening algorithm, and the partial order described in Section \ref{subsec:DegreeRHeadsPartialOrder}. 

\begin{remark}
\label{remark:lactionSameHead}
Let $i\in R_Q$, $\tau \in H_w$. Suppose $X_{\pm i}(\pl_{\tau})$ is nonzero. Then $\pl_{s_{\alpha_i}\tau}=X_{\pm i}(\pl_{\tau})$. In the case of $X_{i}$ we saw that $s_{\alpha_i}\tau$ is obtained from $\tau$ by replacing $i$ with $i+1$, and in the case of $X_{-i}$ we saw that $s_{\alpha_i}\tau$ is obtained from $\tau$ by replacing $i+1$ with $i$. In both cases $i$ and $i+1$ are in the same block since $i \in R_Q$ which implies that $\Cl_{\tau}=\Cl_{s_{\alpha_i}\tau}$. Thus $\tau$ and $s_{\alpha_i}\tau$ have the same head.
\end{remark}

\begin{lemma}
\label{lemma:deg2StraightenPartialOrder}
Let $\tau =(i_1,\ldots,i_d),\phi=(j_1,\ldots,j_d) \in H_w$ with $\tau \ngeq \phi$ in the Bruhat order. If
\begin{center}
$\pl_{\tau}\pl_{\phi} = \displaystyle \sum_{\sigma \in [\tau,\phi]} \pm \pl_{\alpha^\sigma}\pl_{\beta^\sigma}$
\end{center}
is the expression for $\pl_{\tau}\pl_{\phi}$ after a single shuffle, then for every $\sigma \in [\tau,\phi]$ such that $\pl_{\alpha^\sigma} \neq 0$ and $\pl_{\beta^\sigma} \neq 0$ we have:
\begin{enumerate}
\item $\alpha^\sigma > \tau$ and $\beta^\sigma < \phi$.
\item The heads $\theta_{\beta^\sigma}=\theta_{\phi}$ if and only if $\theta_{\alpha^\sigma}=\theta_{\tau}$. Otherwise $\theta_{\alpha^\sigma}>\theta_{\tau}$ and $\theta_{\beta^\sigma}<\theta_{\phi}$.
\end{enumerate}
\end{lemma}
\begin{proof}
(a) The fact that $\alpha^\sigma > \tau$ and $\beta^\sigma < \phi$ is implied by the details of the straightening process discussed in Section \ref{subsec:straighteningAlgorithm}. 

\noindent(b) Fix a $\sigma \in [\tau,\phi]$. Then we have that 
\begin{center}
$\alpha^\sigma = (\sigma(i_1),\ldots,\sigma(i_t),i_{t+1},\ldots,i_d)\uparrow$

$\beta^\sigma = (j_{1},\ldots,j_{t-1}, \sigma(j_t),\ldots,\sigma(j_d))\uparrow$
\end{center}

Suppose that $\theta_{\beta^\sigma}=\theta_{\phi}$. By Corollary \ref{corollary:classDeterminesHeadEquality} this implies that $\Cl_{\beta^{\sigma}}=\Cl_{\phi}$, that is the class of $(j_1,\ldots,j_d)$ equals the class of $(j_{1},\ldots,j_{t-1}, \sigma(j_t),\ldots,\sigma(j_d))\uparrow$. But this implies that the class of $(j_t,\ldots,j_d)$ equals the class of $(\sigma(j_t),\ldots,\sigma_{1}(j_d))$. Thus these two sequences have the same number of entries in $\Bl_{L,k}$ for $1 \leq k \leq \dprl{L}$.

Suppose now that the class of $(\sigma(i_1),\ldots,\sigma(i_t))\uparrow$ does not equal the class of $(i_1,\ldots,i_t)$. This implies that these two sequences have a different number of entries in $\Bl_{L,j}$ for some $1 \leq j \leq \dprl{L}$. But this, combined with the fact that $(j_t,\ldots,j_d)$ and $(\sigma(j_t),\ldots,\sigma(j_d))$ have the same number of entries in $\Bl_{L,j}$, means that $\{\sigma(i_1),\ldots,\sigma(i_t),\sigma(j_t),\ldots,\sigma(j_d)\}$ and $\{i_1,\ldots,i_t,j_t,\ldots,j_d\}$ have a different number of entries in $\Bl_{L,j}$. This is a contradiction of the fact that the set $\{\sigma(i_1),\ldots,\sigma(i_t),\sigma(j_t),\ldots,\sigma(j_d)\}$ is a permutation of the set $\{i_1,\ldots,i_t,j_t,\ldots,j_d\}$.

Thus the class of $(\sigma(i_1),\ldots,\sigma(i_t))\uparrow$ equals the class of $(i_1,\ldots,i_t)$. This implies the class of $\alpha^{\sigma}=(i_{1},\ldots,i_{t-1}, \sigma(i_t),\ldots,\sigma(i_d))\uparrow$ equals the class of $\tau = (i_1,\ldots,i_d)$. But by Corollary \ref{corollary:classDeterminesHeadEquality} this implies $\theta_{\alpha^{\sigma}}=\theta_{\tau}$.

The converse follows by an analogous argument.

To finish the proof of this lemma we note that by Proposition \ref{proposition:PartitionofDegRStandard}, $\beta^\sigma < \phi$ implies that $\theta_{\beta^\sigma}\leq \theta_{\phi}$ and $\alpha^\sigma > \tau$ implies that $\theta_{\alpha^\sigma} \geq \theta_{\tau}$. So if $\theta_{\beta^\sigma} \neq\theta_{\phi}$ and $\theta_{\alpha^\sigma} \neq\theta_{\tau}$ we must have $\theta_{\beta^\sigma} < \theta_{\phi}$ and $\theta_{\alpha^\sigma} > \theta_{\tau}$.
\end{proof}

\begin{theorem}
\label{theorem:degrStraightenPartialOrder}
Let $\tau_1,\ldots,\tau_r \in H_w$  and suppose
\begin{center}
$\pl_{\tau_1} \cdots \pl_{\tau_r} = \displaystyle \sum_{\pl_{\gamma_1} \cdots \pl_{\gamma_r} \in \Std_r} A_{\gamma_1,\ldots,\gamma_r}\pl_{\gamma_1} \cdots \pl_{\gamma_r}$ with $A_{\gamma_1,\ldots,\gamma_r}\in \mathbb{C}$
\end{center}
is the expression for $\pl_{\tau_1} \cdots \pl_{\tau_r}$ as a sum of standard monomials, then $(\theta_{\gamma_1},\ldots,\theta_{\gamma_r}) \geq_{str} (\theta_{\tau_1},\ldots,\theta_{\tau_r})$ for all $\pl_{\gamma_1} \cdots \pl_{\gamma_r} \in \Std_r$ such that $A_{\gamma_1,\ldots,\gamma_r}\neq 0$.
\end{theorem}
\begin{proof}
The straightening of a degree r nonstandard monomial is the result of inductively applying shuffles. Thus we need only show that after applying a single shuffle to an arbitrary nonstandard monomial the inequality holds for all the heads of the monomials resulting from the shuffle.

Let $\pl_{\phi_{1}} \cdots \pl_{\phi_{r}}$ be a degree r nonstandard  monomial. Since $\pl_{\phi_{1}} \cdots \pl_{\phi_{r}}$ is not standard there is a $k$, such that $\phi_{k} \ngeq \phi_{k + 1}$. After a single shuffle we have
\begin{center}
$\pl_{\phi_{1}} \cdots \pl_{\phi_{r}} = \displaystyle \sum_{\sigma \in [\phi_k,\phi_{k+1}]} \pm \pl_{\phi_1} \cdots \pl_{\phi_{k-1}} \pl_{\alpha^\sigma} \pl_{\beta^\sigma} \pl_{\phi_{k+2}} \cdots \pl_{\phi_r}$.
\end{center}

By Lemma \ref{lemma:deg2StraightenPartialOrder} we know for all $\sigma \in [\phi_k,\phi_{k+1}]$ that either $\theta_{\beta^\sigma}=\theta_{\phi_{k+1}}$ and $\theta_{\alpha^\sigma}=\theta_{\phi_{k}}$ or $\theta_{\alpha^\sigma}>\theta_{\phi_k}$ and $\theta_{\beta^\sigma}<\theta_{\phi_{k+1}}$. In either case, we have $(\theta_{\phi_1},\ldots,\theta_{\phi_{k-1}}, \theta_{\alpha^\sigma}, \theta_{\beta^\sigma}, \theta_{\phi_{k+2}},\ldots, \theta_{\phi_r}) \geq_{str} (\theta_{\phi_1},\ldots,\theta_{\phi_r})$.
\end{proof}

\begin{corollary}
\label{corollary:StdThetaLStable}
Let $\underline{\theta} \in \He_{L,r}^{std}$. Then $\spann{\Std_{\underline{\theta}}^{\geq_{str}}}$ and $\spann{\Std_{\underline{\theta}}^{>_{str}}}$ are $L$-stable.
\end{corollary}
\begin{proof}
Let $\pl_{\tau_1} \cdots \pl_{\tau_r} \in \Std_{\underline{\theta}}^{\geq_{str}}$. Let $i \in R_Q$, then by definition
\begin{center}
$X_{\pm i}(\pl_{\tau_1} \cdots \pl_{\tau_r}) = \displaystyle \sum_{j=1}^{r} \pl_{\tau_1} \cdots X_{\pm i}(\pl_{\tau_j}) \cdots \pl_{\tau_r}$.
\end{center}

By Remark \ref{remark:lactionSameHead}, all the monomials on the right hand side that are not equal to zero have heads equal to $(\theta_{\tau_1},\ldots,\theta_{\tau_r}) =_{str} \underline{\theta}$. By Theorem \ref{theorem:degrStraightenPartialOrder}, after straightening the monomials on the right hand side, all the resulting terms will be sums of standard monomials with heads that are greater than or equal, in the partial order $\geq_{str}$, to $(\theta_{\tau_1},\ldots,\theta_{\tau_r})$, which equals $\underline{\theta}$. Thus they will be in $\spann{\Std_{\underline{\theta}}^{\geq_{str}}}$. Thus $\spann{\Std_{\underline{\theta}}^{\geq_{str}}}$ is $\mathfrak{l}$-stable, which implies it is $L$-stable. The same argument shows that $\spann{\Std_{\underline{\theta}}^{>_{str}}}$ is $L$-stable.

\end{proof}

\subsection{The skew semistandard tableau associated to a degree r standard monomial}
\label{subsec:SkewTableauxofStdMon}

Given a degree r standard monomial we would like to associate it to a collection of skew semistandard Young tableau.

Let $\underline{\tau}=(\tau_1,\ldots,\tau_r)$ such that $\pl_{\tau_1} \cdots \pl_{\tau_r} \in \Std_r$. Define the semistandard tableaux $\tab_{\underline{\tau}}$ on the diagram $(r^{d})$ by letting the columns of $\tab_{\underline{\tau}}$ correspond to the $\tau_i$ in $\underline{\tau}$, but with their order reversed. Thus the standardness of $\underline{\tau}$ implies that this tableaux is semistandard. 

\begin{example}
\label{example:skewtableauxFirst}
Let $d=3$ and $N=9$. Consider $w=(3,6,9)\in W^{P_{3}}$. Suppose we have $\underline{\tau}=((3,5,9),(2,3,8),(1,2,4))$. Then
\begin{center}
\ytableausetup{centertableaux}
$\tab_{\underline{\tau}}=$\begin{ytableau}
1 & 2 & 3 \\
2 & 3 & 5 \\
4 & 8 & 9
\end{ytableau}
\end{center}
\end{example}

\begin{definition}
\label{definition:skewsubtab}
We would now like to associate a skew semistandard tableaux $\tab_{\underline{\tau}}^{(k)}$ to $\underline{\tau}$ for each $\Bl_{L,k}$. To do this we start by fixing a $k$, $1\leq k \leq \dprl{L}$. Then, define $\tab_{\underline{\tau}}^{(k)}$ to be the skew semistandard tableaux created by deleting all boxes with values not in $\Bl_{L,k}$, subtracting $a_{k-1}$ from every remaining box, and then deleting all empty rows and columns from the tableaux. Then $\tab_{\underline{\tau}}^{(k)}$ will have boxes with values ranging from 1 to $a_k - a_{k-1}=N_k$. 
\end{definition}

It is not immediately apparent that such an operation will result in $\tab_{\underline{\tau}}^{(k)}$ having a shape that is a skew diagram. However the only way that the shape could fail to be a skew diagram is if one of two possibilities occur:
\begin{enumerate}[label=(\roman*)]
\item For some $i<j$, the maximum column index containing a value in $\Bl_{L,k}$ in row i is less than the maximum column index containing a value in $\Bl_{L,k}$ in row j.

\item For some $i<j$, the minimum column index containing a value in $\Bl_{L,k}$ in row i is less than the minimum column index containing a value in $\Bl_{L,k}$ in row j.

\end{enumerate}

Verifying that neither of these can occur is a simple exercise.

\begin{definition}
\label{definition:skewsubtabpartitions}
Thus the shape of $\tab_{\underline{\tau}}^{(k)}$ is of the form $\nicefrac{\lambda_{\underline{\tau}}^{(k)}}{\mu_{\underline{\tau}}^{(k)}}$ for some partitions $\mu_{\underline{\tau}}^{(k)},$ $\lambda_{\underline{\tau}}^{(k)}$ with $\mu_{\underline{\tau}}^{(k)} \subseteq \lambda_{\underline{\tau}}^{(k)}$. We will always require that $\mu_{\underline{\tau}}^{(k)},$ $\lambda_{\underline{\tau}}^{(k)}$ are the unique choice of partitions so that the skew diagram $\nicefrac{\lambda_{\underline{\tau}}^{(k)}}{\mu_{\underline{\tau}}^{(k)}}$ has no empty rows or columns.
\end{definition}

The total number of boxes in the skew partitions $\nicefrac{\lambda_{\underline{\tau}}^{(k)}}{\mu_{\underline{\tau}}^{(k)}}$ for $1\leq k \leq \dprl{L}$ is equal to the number of boxes in the tableaux $\tab_{\underline{\tau}}$. Thus we have
\begin{equation}
\label{equation:sumofboxesallpartitions}
(|\lambda_{\underline{\tau}}^{(1)}|-|\mu_{\underline{\tau}}^{(1)}|)+\cdots + (|\lambda_{\underline{\tau}}^{(\dprl{L})}|-|\mu_{\underline{\tau}}^{(\dprl{L})}|)=rd
\end{equation}

\begin{example}
\label{example:skewtableaux}
Let us construct the partitions and tableaux associated to the standard monomial $\pl_{\tau_1}\pl_{\tau_2}\pl_{\tau_3}$ with $\underline{\tau}:=(\tau_1,\tau_2,\tau_3)=((3,5,9),(2,3,8),(1,2,4))$ as in Example \ref{example:skewtableauxFirst}. In Example \ref{example:blocks}, we saw that choosing $Q=Q_w$ gave us $\Bl_{L,1}=(1,2,3)$, $\Bl_{L,2}=(4,5,6)$, and $\Bl_{L,3}=(7,8,9)$. Then 
\begin{center}
\ytableausetup{centertableaux}
$\tab_{\underline{\tau}}=$\begin{ytableau}
1 & 2 & 3 \\
2 & 3 & 5 \\
4 & 8 & 9
\end{ytableau}
\;
\end{center}
and deleting boxes from different blocks gives the three tableau

\begin{center}
\;
\begin{ytableau}
1 & 2 & 3\\
2 & 3
\end{ytableau}
$\qquad$
\begin{ytableau}
\none & \none & \none \\
\none & \none & 5 \\
4 & \none & \none
\end{ytableau}
$\qquad$
\begin{ytableau}
\none & \none & \none \\
\none & \none & \none \\
\none & 8 & 9
\end{ytableau}
\;
\end{center}

\noindent subtracting $a_{k-1}$ from each, respectively, gives

\begin{center}
\;
\begin{ytableau}
1 & 2 & 3\\
2 & 3
\end{ytableau}
$\qquad$
\begin{ytableau}
\none & \none & \none \\
\none & \none & 2 \\
1 & \none & \none
\end{ytableau}
$\qquad$
\begin{ytableau}
\none & \none & \none \\
\none & \none & \none \\
\none & 2 & 3
\end{ytableau}
\;
\end{center}

\noindent and finally deleting empty rows and columns gives

\begin{center}
\;
$\tab_{\underline{\tau}}^{(1)}=$\begin{ytableau}
1 & 2 & 3\\
2 & 3
\end{ytableau}
$\qquad \tab_{\underline{\tau}}^{(2)}=$
\begin{ytableau}
\none & 2 \\
1 & \none 
\end{ytableau}
$\qquad \tab_{\underline{\tau}}^{(3)}=$
\begin{ytableau}
2 & 3
\end{ytableau}
\end{center}

\noindent with $\nicefrac{\lambda_{\underline{\tau}}^{(1)}}{\mu_{\underline{\tau}}^{(1)}}=(3,2)/\emptypart$, $\nicefrac{\lambda_{\underline{\tau}}^{(2)}}{\mu_{\underline{\tau}}^{(2)}}=(2,1)/(1)$, and $\nicefrac{\lambda_{\underline{\tau}}^{(3)}}{\mu_{\underline{\tau}}^{(3)}}=(2)/\emptypart$.

%
\end{example}

\begin{lemma}
\label{lemma:skewTableauxSameHead}
Let $\underline{\theta} := (\theta_1,\ldots,\theta_r)\in \He_{L,r}^{std}$ be a standard degree r head. If $\underline{\tau} := (\tau_1,\ldots,\tau_r),\underline{\gamma}:= (\gamma_1,\ldots,\gamma_r)$ are two sequence such that $(\theta_{\tau_1},\ldots,\theta_{\tau_r})=(\theta_{\gamma_1},\ldots,\theta_{\gamma_r})=\underline{\theta}$, then $\lambda_{\underline{\tau}}^{(k)}=\lambda_{\underline{\gamma}}^{(k)}=\lambda_{\underline{\theta}}^{(k)}$ and $\mu_{\underline{\tau}}^{(k)}=\mu_{\underline{\gamma}}^{(k)}=\mu_{\underline{\theta}}^{(k)}$, for $1\leq k \leq \dprl{L}$ 
\end{lemma}
\begin{proof}
By Proposition \ref{proposition:PartitionofHasse} we know that $\Cl_{\theta_i}=\Cl_{\tau_i}=\Cl_{\gamma_i}$ for all $1\leq i \leq r$. Thus for all $1\leq i \leq r$, $1\leq j \leq d$, and $1\leq k \leq \dprl{L}$ we have the $j$th entry of $\theta_{i}$ is in $\Bl_{L,k} \iff$  the $j$th entry of $\tau_{i}$ is in $\Bl_{L,k}\iff$  the $j$th entry of $\gamma_{i}$ is in $\Bl_{L,k}$. As the shape of the associated tableau depend only on the block membership of the entries we are done.
\end{proof}


\begin{definition}
\label{definition:isomorphismSkewTableaux}
Let $\underline{\theta} \in \He_{L,r}^{std}$ be a degree r head. Set $V_k := \mathbb{C}^{N_k}$, $1\leq k \leq \dprl{L}$ . Recall that $\wmodel{\tab}{}\coloneqq e_{\tab}\Upsilon_W^{\lambda / \mu} \in \mathbb{W}^{\lambda / \mu}(\mathbb{C}^{N})$. Define a vector space map on the basis elements of $\spann{\Std_{\underline{\theta}}}$ as follows:
\begin{center}
$\displaystyle \Psi_{\underline{\theta}}:\spann{\Std_{\underline{\theta}}}\longrightarrow \mathbb{W}_{\underline{\theta}} := \mathbb{W}^{\nicefrac{\lambda_{\underline{\theta}}^{(1)}}{\mu_{\underline{\theta}}^{(1)}}}(V_1) \otimes \cdots \otimes \mathbb{W}^{\nicefrac{\lambda_{\underline{\theta}}^{(\dprl{L})}}{\mu_{\underline{\theta}}^{(\dprl{L})}}}(V_{\dprl{L}})$

$\pl_{\tau_1}\cdots \pl_{\tau_r} \longmapsto \wmodel{\tab_{\underline{\tau}}^{(1)}}{}\otimes \cdots \otimes \wmodel{\tab_{\underline{\tau}}^{(\dprl{L})}}{}$

\end{center}
where $\underline{\tau} := (\tau_1,\ldots,\tau_r)$.

Lemma \ref{lemma:skewTableauxSameHead}, the definition of the semistandard tableaux $\tab_{\underline{\tau}}^{(k)}$ for $1\leq k \leq \dprl{L}$, and Theorem \ref{theorem:ssytBasisWeylModule} give that this map is well defined and takes basis vectors to basis vectors. 

Subsequently when we refer to the skew Weyl modules in the above tensor product we will write them as $\mathbb{W}^{\nicefrac{\lambda_{\underline{\theta}}^{(k)}}{\mu_{\underline{\theta}}^{(k)}}}$, omitting the $(V_{k})$, so long as no confusion may arise from doing so.
\end{definition}

\begin{proposition}
\label{proposition:isomorphismSkewTableaux}
Let $\underline{\theta} := (\theta_1,\ldots,\theta_r)\in \He_{L,r}^{std}$ be a degree r head. The map $\Psi_{\underline{\theta}}$ is a vector space isomorphism.
\end{proposition}
\begin{proof}
We describe a map $\Phi_{\underline{\theta}}$ going from $\mathbb{W}_{\underline{\theta}}$ to $\spann{\Std_{\underline{\theta}}}$. Let $\wmodel{\tab^{(1)}}{}\otimes \cdots \otimes \wmodel{\tab^{(\dprl{L})}}{}$ be a basis vector of $\mathbb{W}_{\underline{\theta}}$. Then each $\tab^{(k)}$ is a SSYT on the skew diagram $\nicefrac{\lambda_{\underline{\theta}}^{(k)}}{\mu_{\underline{\theta}}^{(k)}}$.

We now perform the reverse of the process described in Definition \ref{definition:skewsubtab}. Firstly, if any empty rows or columns were deleted to form $\tab_{\underline{\theta}}^{(k)}$ we add those empty rows and columns back to $\tab^{(k)}$. We then add $a_{k-1}$ to each box of $\tab^{(k)}$. Finally, we combine the boxes from $\tab^{(1)},\ldots,\tab^{(\dprl{L})}$ into the rectangular tableaux $\tab$ of shape $(r^d)$. 

When comparing two boxes of $\tab$ that are in the same block, the requirements for semistandardness are fulfilled since the individual tableaux associated with each block is a SSYT. When comparing two boxes of $\tab$ that are not in the same block, if these two entries violated semistandardness then the same boxes in $\tab_{\underline{\theta}}$ would violate semistandardness. Thus $\tab$ is a SSYT.

Finally, we define $\underline{\tau}=(\tau_1,\ldots,\tau_r)$ by letting the columns of $\tab_{\underline{\tau}}$ correspond to the $\tau_i$ in $\underline{\tau}$, but with their order reversed. The fact that $\tab$ is semistandard implies that $\underline{\tau}$ is standard. The fact that $\tab(i,j) \in \Bl_{L,k} \iff \tab_{\underline{\theta}}(i,j) \in \Bl_{L,k}$ for all $k$ implies that the $j$th entry in $\theta_{i}$ is in $\Bl_{L,k} \iff$ the $j$th entry in $\tau_{i}$ is in $\Bl_{L,k}$, which implies that $\theta_i=\theta_{\tau_i}$. That is $\pl_{\tau_1}\cdots \pl_{\tau_r} \in \spann{\Std_{\underline{\theta}}}$.

But then it is clear that the map $\Phi_{\underline{\theta}}$ is well defined and in fact the inverse of  $\Psi_{\underline{\theta}}$. This can be trivially verified on the basis vectors. And thus $\Psi_{\underline{\theta}}$ is a vector space isomorphism.
\end{proof}

For $\pl_{\tau_1} \cdots \pl_{\tau_r}\in\Std_{\underline{\theta}}^{\geq_{str}}$ let $\overline{\pl_{\tau_1} \cdots \pl_{\tau_r}}$ denote its class in $\spann{\Std_{\underline{\theta}}^{\geq_{str}}} / \spann{\Std_{\underline{\theta}}^{>_{str}}}$ under the canonical quotient map 
\begin{center}
$\spann{\Std_{\underline{\theta}}^{\geq_{str}}} \rightarrow \spann{\Std_{\underline{\theta}}^{\geq_{str}}} / \spann{\Std_{\underline{\theta}}^{>_{str}}}$.
\end{center}
Then $\{ \overline{\pl_{\tau_1} \cdots \pl_{\tau_r}} \mid \pl_{\tau_1} \cdots \pl_{\tau_r}\in \Std_{\underline{\theta}}\}$ is a basis for $\spann{\Std_{\underline{\theta}}^{\geq_{str}}} / \spann{\Std_{\underline{\theta}}^{>_{str}}}$. 

\noindent Let $\{(\wmodel{\tab_{\underline{\tau}}^{(1)}}{}\otimes \cdots \otimes \wmodel{\tab_{\underline{\tau}}^{(\dprl{L})}}{})^{*} \mid \underline{\tau}=(\tau_1,\ldots,\tau_r) \textrm{ such that } \pl_{\tau_1} \cdots \pl_{\tau_r} \in \Std_{\underline{\theta}}\}$ be the basis of $\mathbb{W}^{*}_{\underline{\theta}}$ dual to the basis of $\mathbb{W}_{\underline{\theta}}$ given by $\{\wmodel{\tab_{\underline{\tau}}^{(1)}}{}\otimes \cdots \otimes \wmodel{\tab_{\underline{\tau}}^{(\dprl{L})}}{} \mid \underline{\tau}=(\tau_1,\ldots,\tau_r) \textrm{ such that } \pl_{\tau_1} \cdots \pl_{\tau_r} \in \Std_{\underline{\theta}}\}$.

\begin{definition}
\label{definition:isomorphismModuleSkewTableaux}
The isomorphism $\Psi_{\underline{\theta}}$ induces a vector space map 

\begin{center}
$\overline{\Psi}_{\underline{\theta}}:\spann{\Std_{\underline{\theta}}^{\geq_{str}}} / \spann{\Std_{\underline{\theta}}^{>_{str}}} \longrightarrow \mathbb{W}^{*}_{\underline{\theta}}$
\end{center}

\begin{center}
$\overline{\pl_{\tau_1} \cdots \pl_{\tau_r}}  \longmapsto \left( \wmodel{\tab_{\underline{\tau}}^{(1)}}{} \otimes \cdots \otimes \wmodel{\tab_{\underline{\tau}}^{(\dprl{L})}}{} \right)^{*}$ 
\end{center}
for $\pl_{\tau_1} \cdots \pl_{\tau_r} \in \Std_{\underline{\theta}}$. 
\end{definition}

\begin{proposition}
\label{proposition:isomorphismModuleSkewTableaux}
The map $\overline{\Psi}_{\underline{\theta}}$ is an isomorphism of vector spaces.
\end{proposition}
\begin{proof}
That this map is an isomorphism follows from the fact that it is the composition of three isomorphisms. The first is from $\spann{\Std_{\underline{\theta}}^{\geq_{str}}} / \spann{\Std_{\underline{\theta}}^{>_{str}}}$ to $\spann{\Std_{\underline{\theta}}}$(cf. Remark \ref{remark:relationStdTheta}), the second is from $\spann{\Std_{\underline{\theta}}}$ to $W_{\underline{\theta}}$(cf. Proposition \ref{proposition:isomorphismSkewTableaux}), and the third is the canonical isomorphism from $W_{\underline{\theta}}$ to $W^{*}_{\underline{\theta}}$.
\end{proof}

Our goal now is to show that $\mathbb{W}^{*}_{\underline{\theta}}$ has a canonical $L$-module structure and then use the map $\overline{\Psi}_{\underline{\theta}}$ to relate its $L$-module structure to the $L$-module structure of $\spann{\Std_{\underline{\theta}}^{\geq_{str}}} / \spann{\Std_{\underline{\theta}}^{>_{str}}}$(cf. Corollary \ref{corollary:StdThetaLStable}).

\subsection{The \texorpdfstring{$\mathfrak{l}$}{l}-module structure of \texorpdfstring{$\mathbb{W}_{\underline{\theta}}$}{W} and the implications for our main theorem}
\label{subsec:lactionMainTheorem}

Recall that $L=\GL_{N_1}\times \cdots \times \GL_{N_{\dprl{L}}}$ and 

\begin{center}
$\mathbb{W}^{*}_{\underline{\theta}}=(\mathbb{W}^{\nicefrac{\lambda_{\underline{\theta}}^{(1)}}{\mu_{\underline{\theta}}^{(1)}}} \otimes \cdots \otimes \mathbb{W}^{\nicefrac{\lambda_{\underline{\theta}}^{(\dprl{L})}}{\mu_{\underline{\theta}}^{(\dprl{L})}}})^{*}\cong(\mathbb{W}^{\nicefrac{\lambda_{\underline{\theta}}^{(1)}}{\mu_{\underline{\theta}}^{(1)}}})^{*} \otimes \cdots \otimes (\mathbb{W}^{\nicefrac{\lambda_{\underline{\theta}}^{(\dprl{L})}}{\mu_{\underline{\theta}}^{(\dprl{L})}}})^{*}$.
\end{center}

Each $\mathbb{W}^{\nicefrac{\lambda_{\underline{\theta}}^{(k)}}{\mu_{\underline{\theta}}^{(k)}}}$ is a Weyl Module and thus has a canonical $\GL_{N_k}$-module structure. Thus $(\mathbb{W}^{\nicefrac{\lambda_{\underline{\theta}}^{(k)}}{\mu_{\underline{\theta}}^{(k)}}})^{*}$ has an induced $\GL_{N_k}$-module structure. The $L$-module structure for $\mathbb{W}^{*}_{\underline{\theta}}$ is simply given by the induced product structure. 

Let $T_L\subset T$ be the maximal torus in $L$. In Proposition \ref{proposition:isomorphismModuleSkewTableaux} we exhibited a vector space ismorphism from $\spann{\Std_{\underline{\theta}}^{\geq_{str}}} / \spann{\Std_{\underline{\theta}}^{>_{str}}}$ to $\mathbb{W}^{*}_{\underline{\theta}}$ that takes the $T_L$ weight vector $\overline{\pl_{\tau_1}\cdots \pl_{\tau_r}}$, $\pl_{\tau_1} \cdots \pl_{\tau_r} \in \Std_{\underline{\theta}}\;$, to the $T_L$ weight vector $(\wmodel{\tab_{\underline{\tau}}^{(1)}}{}\otimes \cdots \otimes \wmodel{\tab_{\underline{\tau}}^{(\dprl{L})}}{})^{*}$. We will use this to relate the characters of these two $L$-modules.

Since the map $\spann{\Std_{\underline{\theta}}^{\geq_{str}}}\rightarrow \spann{\Std_{\underline{\theta}}^{\geq_{str}}} / \spann{\Std_{\underline{\theta}}^{>_{str}}}$ is $L$-equivariant, and thus $T_L$-equivariant, we have that 
\begin{center}
$wt(\overline{\pl_{\tau_1}\cdots \pl_{\tau_r}})=wt(\pl_{\tau_1}\cdots \pl_{\tau_r})=wt(\pl_{\tau_1})+\cdots+wt(\pl_{\tau_r})$.
\end{center}

As discussed in Remark \ref{remark:weightPlucker} we have that the weight of $\pl_{\tau}$ is given by the sequence $\chi_{\tau}:=(\chi_1,\ldots,\chi_{N})$ where
\begin{center}
$\chi_i \coloneqq \left\{
\begin{array}{ll}
      -1 & i\in \tau\\
      0 & i\notin \tau 
\end{array} 
\right.$
\end{center}
for all $1\leq i \leq N$. Let $n_{\underline{\tau}}^{(i)}$ equal the number times the value $i$ appears in $\underline{\tau}$. Combining these results we have that $wt(\overline{\pl_{\tau_1}\cdots \pl_{\tau_r}})=(-n_{\underline{\tau}}^{(1)},\ldots,-n_{\underline{\tau}}^{(N)})$. And thus
\begin{equation}
\label{eq:charStdTheta}
char(\spann{\Std_{\underline{\theta}}^{\geq_{str}}} / \spann{\Std_{\underline{\theta}}^{>_{str}}})=\displaystyle \sum_{\substack{\underline{\tau}=(\tau_1,\ldots,\tau_r) \\ \textrm{ s.t. } \pl_{\tau_1} \cdots \pl_{\tau_r} \in \Std_{\underline{\theta}}}}e^{(-n_{\underline{\tau}}^{(1)},\ldots,-n_{\underline{\tau}}^{(N)})}\;(\textrm{cf. Section }\ref{subsec:glN})
\end{equation}

\noindent Regarding $char(\mathbb{W}^{*}_{\underline{\theta}})$ we have that
\begin{center}
$wt((\wmodel{\tab_{\underline{\tau}}^{(1)}}{}\otimes \cdots \otimes \wmodel{\tab_{\underline{\tau}}^{(\dprl{L})}}{})^{*})=-wt(\wmodel{\tab_{\underline{\tau}}^{(1)}}{}\otimes \cdots \otimes \wmodel{\tab_{\underline{\tau}}^{(\dprl{L})}}{})$
\end{center}
for all $\underline{\tau}=(\tau_1,\ldots,\tau_r) \textrm{ such that } \pl_{\tau_1} \cdots \pl_{\tau_r} \in \Std_{\underline{\theta}}$. Let $\xi_{\underline{\tau}}:=(\xi_1,\ldots,\xi_{N})$ be the weight of $T_L$ weight vector $\wmodel{\tab_{\underline{\tau}}^{(1)}}{}\otimes \cdots \otimes \wmodel{\tab_{\underline{\tau}}^{(\dprl{L})}}{}$. Then $\xi_i$ is equal to the number of entries in $\tab_{\underline{\tau}}^{(k)}$ equal to $i-a_{k-1}$ for the unique $k$ such that $i \in\left\{a_{k-1} + 1,\ldots, a_{k} \right\}$. But $n_{\underline{\tau}}^{(i)}$ is the number of entries in $\tab_{\underline{\tau}}^{(k)}$ equal to $i-a_{k-1}$. Thus we have that $wt((\wmodel{\tab_{\underline{\tau}}^{(1)}}{}\otimes \cdots \otimes \wmodel{\tab_{\underline{\tau}}^{(\dprl{L})}}{})^{*})=(-n_{\underline{\tau}}^{(1)},\ldots,-n_{\underline{\tau}}^{(N)})$ which implies
\begin{equation}
\label{eq:charWTheta}
char(\mathbb{W}^{*}_{\underline{\theta}}) = \displaystyle \sum_{\substack{\underline{\tau}=(\tau_1,\ldots,\tau_r) \\ \textrm{ s.t. } \pl_{\tau_1} \cdots \pl_{\tau_r} \in \Std_{\underline{\theta}}}}e^{(-n_{\underline{\tau}}^{(1)},\ldots,-n_{\underline{\tau}}^{(N)})}.
\end{equation}

\begin{proposition}
\label{proposition:psiLModuleIsomorphism}
The $L$-modules $\spann{\Std_{\underline{\theta}}^{\geq_{str}}} / \spann{\Std_{\underline{\theta}}^{>_{str}}}$ and $\mathbb{W}^{*}_{\underline{\theta}}$ are isomorphic.
\end{proposition}
\begin{proof}
We have that $char(\spann{\Std_{\underline{\theta}}^{\geq_{str}}} / \spann{\Std_{\underline{\theta}}^{>_{str}}}) = char(\mathbb{W}^{*}_{\underline{\theta}})$ by \eqref{eq:charStdTheta} and \eqref{eq:charWTheta}. Since $\mathbb{C}[X(w)]$ is a quotient of the rational $\GL_N$ representation $\mathbb{C}[\mathbb{P}(\bigwedge^d \mathbb{C}^N)]$ by an $L$-stable ideal, it is a rational $L$-representation, and thus any $L$-subrepresentation is a rational $L$-representation. And the quotient of two rational representations, $\spann{\Std_{\underline{\theta}}^{\geq_{str}}} / \spann{\Std_{\underline{\theta}}^{>_{str}}}$ is a rational $L$-representation. Since $\mathbb{W}^{*}_{\underline{\theta}}$ is the dual of the tensor product of polynomial representations, it is a rational $L$-representation. It follows from Corollary \ref{corollary:ratCharIso} that two rational $L$-representations are isomorphic if and only if their characters are equal, and hence we have our desired result.
\end{proof}

\begin{theorem}
\label{theorem:mainDecompositionTheorem}
Let $\underline{\theta} \in \He_{L,r}^{std}$. There exists a $L$-submodule $U_{\underline{\theta}} \subseteq \spann{\Std_r}$ such that we have the following $L$-module isomorphisms:
\begin{enumerate}
\item $\spann{\Std_{\underline{\theta}}^{\geq_{str}}} = U_{\underline{\theta}} \oplus \spann{\Std_{\underline{\theta}}^{>_{str}}}$.
\item $\displaystyle \spann{\Std_r} = \bigoplus_{\underline{\theta} \in \He_{L,r}^{std}} U_{\underline{\theta}}$.
\item $U_{\underline{\theta}} \cong \mathbb{W}^{*}_{\underline{\theta}}$
\end{enumerate}
\end{theorem}
\begin{proof}
\noindent(a) We are in characteristic 0; so $L$ being reductive we have that $L$ is linearly reductive. Thus any $L$-module is completely reducible. This implies $\spann{\Std_{\underline{\theta}}^{\geq_{str}}}$ is completely reducible, and since $\spann{\Std_{\underline{\theta}}^{>_{str}}}$ is a $L$-submodule of $\spann{\Std_{\underline{\theta}}^{\geq_{str}}}$ it must have a $L$-module complement which we denote $U_{\underline{\theta}}$. Thus $\spann{\Std_{\underline{\theta}}^{\geq_{str}}} = U_{\underline{\theta}} \oplus \spann{\Std_{\underline{\theta}}^{>_{str}}}$ as $L$-modules.

\noindent(b) Note that the $L$-module complement  of  $\spann{\Std_{\underline{\theta}}^{>_{str}}}$ in $\spann{\Std_{\underline{\theta}}^{\geq_{str}}}$ is not unique, and so we must make a choice of a particular complement which we denote $U_{\underline{\theta}}$. Nonetheless, all arguments in this proof work regardless of what choice is made.  We have the following vector space isomorphisms 
\begin{equation}
\label{equation:maindecomp1}
\begin{array}{rll}
U_{\underline{\theta}} & \cong\;\; \spann{\Std_{\underline{\theta}}^{\geq_{str}}} / \spann{\Std_{\underline{\theta}}^{>_{str}}} & \textrm{ (by (a))} \\
\; & \cong\;\;\spann{Std_{\underline{\theta}}} & \textrm{ (Remark \ref{remark:relationStdTheta})}\\
\end{array}
\end{equation}
And by Corollary \ref{corollary:decompositionStdr} we have that
\begin{equation}
\label{equation:maindecomp2}
\displaystyle \spann{\Std_r}\;\;=\bigoplus_{\underline{\theta} \in \He_{L,r}^{std}} \spann{\Std_{\underline{\theta}}}
\end{equation}
as vector spaces. 

Now consider $U_{\underline{\theta}_1}$ and $U_{\underline{\theta}_2}$ for $\underline{\theta}_1$, $\underline{\theta}_2\in \He_{L,r}^{std}$ with $\underline{\theta}_1 \neq \underline{\theta}_2$. Then we claim that $U_{\underline{\theta}_1} \bigcap U_{\underline{\theta}_2} = \{0\}$. To see why this is the case we have two possibilities to consider.

\noindent\textbf{Case 1:} $\underline{\theta}_1$ and $\underline{\theta}_2$ comparable in the partial order $\geq_{str}$. Then without loss of generality say $\underline{\theta}_1 <_{str} \underline{\theta}_2$. Then  $\spann{\Std_{\underline{\theta}_2}^{\geq_{str}}} \subseteq \spann{\Std_{\underline{\theta}_1}^{>_{str}}}$. But then since $U_{\underline{\theta}_1}$ is an $L$-module complement of $\spann{\Std_{\underline{\theta}_1}^{>_{str}}}$, and $U_{\underline{\theta}_2}\subset \spann{\Std_{\underline{\theta}_2}^{\geq_{str}}}$, this implies $U_{\underline{\theta}_1} \bigcap U_{\underline{\theta}_2} = \{0\}$.

\noindent\textbf{Case 2:} $\underline{\theta}_1$ and $\underline{\theta}_2$ non-comparable in the partial order $\geq_{str}$. This implies
\begin{equation}
\label{equation:maintheorem3}
\Std_{\underline{\theta}_2}^{\geq_{str}}\bigcap\Std_{\underline{\theta}_1}=\emptyset.
\end{equation}

Let $f \in U_{\underline{\theta}_1}$, then since $U_{\underline{\theta}_1} \subset \spann{\Std_{\underline{\theta}_1}^{\geq_{str}}}$ we have $f = \sum A_i f_i$ for some $A_i \in \mathbb{C}$ and $f_i \in \Std_{\underline{\theta}_1}^{\geq_{str}}$. Note that at least one of these $f_i$ is in $Std_{\underline{\theta}_1}$ and appears with nonzero $A_i$, otherwise $f \in \spann{\Std_{\underline{\theta}_1}^{>_{str}}}$, which contradicts the definition of $U_{\underline{\theta}_1}$. So we can rewrite $f=\sum B_j g_j + \sum C_k h_k$ for some $B_j,C_k \in \mathbb{C}$, $g_j \in \Std_{\underline{\theta}_1}$ and $h_k \in \Std_{\underline{\theta}_1}^{>_{str}}$ with not all $B_j$ equal to zero.

Now suppose $f \in U_{\underline{\theta}_2}$. This implies $f = \sum D_i x_i$ for some $D_i \in \mathbb{C}$ and $x_i \in \Std_{\underline{\theta}_2}^{\geq_{str}}$. Combining the two different expressions for $f$ gives 
\begin{equation}
\label{equation:maintheorem4}
\sum B_j g_j = \sum D_i x_i - \sum C_k h_k. 
\end{equation}
We have the $g_j \in \Std_{\underline{\theta}_1}$, $x_i \in \Std_{\underline{\theta}_2}^{\geq_{str}}$, and $h_k \in \Std_{\underline{\theta}_1}^{>_{str}}$. Now by \eqref{equation:maintheorem3} and the fact that
\begin{center}
$\Std_{\underline{\theta}_1}\bigcap\Std_{\underline{\theta}_1}^{>_{str}}=\emptyset$.
\end{center}
this means that in \eqref{equation:maintheorem4} we are writing a linear combination of standard monomials in $\Std_{\underline{\theta}_1}$ as a linear combination of standard monomials not in $\Std_{\underline{\theta}_1}$, hence both sides of \eqref{equation:maintheorem4} should equal zero. Since not all the $B_j$ are zero this is not the case, and thus is a violation of the linear independence of the standard monomials. Thus $f \notin U_{\underline{\theta}_2}$ and we have $U_{\underline{\theta}_1} \bigcap U_{\underline{\theta}_2} = \{0\}$.

Thus the subspace of $\spann{\Std_r}$ that is defined as the sum of all the $U_{\underline{\theta}}$, $\underline{\theta} \in \He_{L,r}^{std}$, is a direct sum, that is
\begin{center}
$\displaystyle \sum_{\underline{\theta} \in \He_{L,r}^{std}} U_{\underline{\theta}}=\displaystyle \bigoplus_{\underline{\theta} \in \He_{L,r}^{std}} U_{\underline{\theta}}$.
\end{center}
But \eqref{equation:maindecomp1} and \eqref{equation:maindecomp2} imply that this subspace is in fact equal to $\spann{\Std_r}$ by dimension considerations. Thus we have 
\begin{center}
$\displaystyle \spann{\Std_r}\;\;=\bigoplus_{\underline{\theta} \in \He_{L,r}^{std}} U_{\underline{\theta}}$
\end{center}
as vector spaces. Now since each $U_{\underline{\theta}}$ is $L$-stable this is in fact an equality of $L$-modules.

\noindent(c) We have that $U_{\underline{\theta}} \cong \spann{\Std_{\underline{\theta}}^{\geq_{str}}} / \spann{\Std_{\underline{\theta}}^{>_{str}}}$ as $L$-modules. And by Proposition \ref{proposition:psiLModuleIsomorphism} we have $\spann{\Std_{\underline{\theta}}^{\geq_{str}}} / \spann{\Std_{\underline{\theta}}^{>_{str}}}$ is isomorphic to $\mathbb{W}^{*}_{\underline{\theta}}$ as $L$-modules.
\end{proof}

\begin{remark}
\label{remark:decompDegree1}
When $\theta \in \He_{L,1}$ $(=\He_L)$ we have that $\spann{\Std_{\theta}}$ is $L$-stable. This can be seen by noting that it will be $\mathfrak{l}$-stable, which follows immediately from the description of the $\mathfrak{l}$-action in Section \ref{subsec:lactionMainTheorem} and Proposition \ref{proposition:PartitionofHasse}. This implies that $U_{\theta} \cong \spann{\Std_{\theta}}$ as $L$-modules.

Further, when $\theta \in \He_{L,1}$ we have that 
\begin{center}
$\mathbb{W}_{\theta} = \mathbb{W}^{(1^{m_1})}\otimes \cdots \otimes \mathbb{W}^{(1^{m_{\dprl{L}}})}$.
\end{center}
for some non-negative integers $m_1,\ldots,m_{\dprl{L}}$(cf. Section \ref{subsec:SkewTableauxofStdMon}).

Thus for such $\theta$ we have that $U_{\theta} \cong \spann{\Std_{\theta}}$ is an irreducible $L$-module.
\end{remark}

\begin{corollary}
\label{corollary:mainDecompositionTheoremIrred}
The ring $\mathbb{C}[X(w)]$ has the following decomposition into irreducible $L$-modules 
\begin{center}
$\displaystyle \bigoplus_{r \geq 1} \bigoplus_{\underline{\theta} \in \He_{L,r}^{std}} \left(\left(\displaystyle \bigoplus_{\nu_{\underline{\theta}}^{(1)}} \left(\left({\mathbb{W}^{\nu_{\underline{\theta}}^{(1)}}}\right)^{*}\right)^{\oplus c_{\mu_{\underline{\theta}}^{(1)},\; \nu_{\underline{\theta}}^{(1)}}^{\lambda_{\underline{\theta}}^{(1)}}}\right) \otimes \cdots \otimes \left(\displaystyle \bigoplus_{\nu_{\underline{\theta}}^{(\dprl{L})}} \left(\left({\mathbb{W}^{\nu_{\underline{\theta}}^{(\dprl{L})}}}\right)^{*}\right)^{\oplus c_{\mu_{\underline{\theta}}^{(\dprl{L})},\; \nu_{\underline{\theta}}^{(\dprl{L})}}^{\lambda_{\underline{\theta}}^{(\dprl{L})}}}\right)\right)$ 
\end{center}
where for all $1\leq k \leq \dprl{L}$: the $\lambda_{\underline{\theta}}^{(k)}$, $\mu_{\underline{\theta}}^{(k)}$ are the partitions defined in Definition \ref{definition:skewsubtabpartitions}, the innermost direct sums are over all partitions $\nu_{\underline{\theta}}^{(k)}$ such that $|\nu_{\underline{\theta}}^{(k)}|=|\lambda_{\underline{\theta}}^{(k)}|-|\mu_{\underline{\theta}}^{(k)}|$, and the $c_{\mu_{\underline{\theta}}^{(k)},\; \nu_{\underline{\theta}}^{(k)}}^{\lambda_{\underline{\theta}}^{(k)}}$ are the Littlewood-Richardson coefficients associated to these partitions.
\end{corollary}
\begin{proof}
The decomposition of $\mathbb{W}_{\underline{\theta}}$ into irreducible $L$-modules may be obtained by using \ref{equation:glnweylmoduledecomp};
\begin{center}
$\begin{array}{rl}
\mathbb{W}_{\underline{\theta}} & =\mathbb{W}^{\nicefrac{\lambda_{\underline{\theta}}^{(1)}}{\mu_{\underline{\theta}}^{(1)}}} \otimes \cdots \otimes \mathbb{W}^{\nicefrac{\lambda_{\underline{\theta}}^{(\dprl{L})}}{\mu_{\underline{\theta}}^{(\dprl{L})}}} \\
\; & = \left(\displaystyle \bigoplus_{\nu_{\underline{\theta}}^{(1)}} \left({\mathbb{W}^{\nu_{\underline{\theta}}^{(1)}}}\right)^{\oplus c_{\mu_{\underline{\theta}}^{(1)},\; \nu_{\underline{\theta}}^{(1)}}^{\lambda_{\underline{\theta}}^{(1)}}}\right) \otimes \cdots \otimes \left(\displaystyle \bigoplus_{\nu_{\underline{\theta}}^{(\dprl{L})}} \left({\mathbb{W}^{\nu_{\underline{\theta}}^{(\dprl{L})}}}\right)^{\oplus c_{\mu_{\underline{\theta}}^{(\dprl{L})},\; \nu_{\underline{\theta}}^{(\dprl{L})}}^{\lambda_{\underline{\theta}}^{(\dprl{L})}}}\right).
\end{array}$
\end{center}
The corollary follows by taking the dual of $\mathbb{W}_{\underline{\theta}}$ and Theorem \ref{theorem:mainDecompositionTheorem}.

\end{proof}

\begin{remark}
\label{remark:howeBranching}
Let $L_w$ be the Levi subgroup of the stabilizer $Q_w$. If $\dprl{L_w}=1$ we have that $L_w=\GL_N$ and $\mathbb{C}[X(w)]_r \cong \left(\mathbb{W}^{(r^{d})}\right)^{*}$. Further, Theorem \ref{theorem:mainDecompositionTheorem} and Corollary \ref{corollary:mainDecompositionTheoremIrred} give the decomposition of $\mathbb{C}[X(w)]_r$ for any $L$ the Levi part of a parabolic subgroup $Q \subset Q_w$. This $L$ is a subgroup of $L_w$, in fact $L=\GL_{N_1} \times \cdots \times \GL_{N_{\dprl{L}}}$ is embedded diagonally in $L_w=\GL_N$. Further $X(w)=\Grass_{d,N}$. Thus as a consequence of our explicit decomposition we get the branching rules for the Weyl module $\left(\mathbb{W}^{(r^{d})}\right)^{*}$ for any $\GL_{N_1} \times \cdots \times \GL_{N_{\dprl{L}}}$ with $N_1 + \cdots + N_{\dprl{L}}=N$ diagonally embedded in $\GL_N$. This branching rule is discussed in much greater generality in \cite{MR2115378} for $\GL_m \times \GL_n$ embedded diagonally in $\GL_{n+m}$. It seems reasonable to expect that further exploration of the cases when $\dprl{L_w}>1$ might yield additional non-trivial branching rules for representations of $\GL_N$.
\end{remark}

\section{Sphericity consequences of the decomposition}
\label{sec:sphericityconsequences}

The decomposition results from the previous section may be used to show that certain classes of Schubert varieties are in fact spherical varieties. For this section fix $d, N$ positive integers with $d<N$. Let $P=P_{d}$ and $w=(\ell_1,\ldots,\ell_{d})\in W^{P}$. Let $L_w$ be the Levi part of the stabilizer $Q_w$ of $X(w)$.

Let $G$ be a connected reductive group with $B_G$ a Borel subgroup. Suppose that $X$ is an irreducible $G$-variety. Then $X$ is a \emph{spherical $G$-variety} if it is normal and it has an open dense $B_G$-orbit(cf. \cite{MR837530}). We wish to relate the sphericity of a projective variety $X\hookrightarrow \mathbb{P}(V)$ and the cone $\widehat{X}$ over $X$.

\begin{proposition}
\label{proposition:coneProjVarMultFreeSphericity}
Let $X$ be projectively normal, namely $\mathbb{C}[X]$, the homogeneous coordinate ring of $X$, is normal. Let $G$ be a connected reductive group acting linearly on $X$, that is $V$ is a $G$-module and the action of $G$ on $X$ is induced from the $G$-action on $\mathbb{P}(V)$. If $\mathbb{C}[X]$ has a multiplicity free decomposition into irreducible $G$-modules, then $\widehat{X},X$ are spherical $G$-varieties.
\end{proposition}
\begin{proof}
We have, by \cite[Theorem 25.1]{MR2797018}, that an affine normal G-variety (G reductive) is a spherical G-variety if and only if the decomposition of its coordinate ring into irreducible G-modules is multiplicity free.  Thus $\widehat{X}$ is a spherical $G$-variety and hence has a dense open $B$-orbit $U$. 

The canonical map $\pi: V \setminus \{0\} \rightarrow \mathbb{P}(V)$ is $G$-equivariant. Denoting the restriction of $\pi$ to $\widehat{X} \setminus \{0\}$ by $\pi'$, we get a $G$-equivariant map $\pi': \widehat{X} \setminus \{0\} \rightarrow X$. The map $\pi'$ is a principal fiber bundle for the action of the multiplicative group $\mathbb{G}_m$, and therefore a geometric quotient. Hence $\pi'$ is an open map and we get that $U':=\pi'(U)$ is open and dense in $X$. As $U'$ is the image of a $B$-orbit under a $G$-equivariant map we have that $U'$ is itself a $B$-orbit. Thus $X$ is a spherical $G$-variety.
\end{proof}

In light of this result, for a Schubert variety $X(w)$, if the decomposition of $\mathbb{C}[X(w)]$ into irreducible $L_w$-modules is multiplicity free, then $X(w)$ is a spherical $L_w$ variety. Using Theorem \ref{theorem:mainDecompositionTheorem}, we will exhibit several classes of Schubert varieties for which this is the case. 

\begin{theorem}
\label{thm:smoothSpherical}
 Let $X(w)$ be a smooth Schubert variety in $\Grass_{d,N}$. Then the decomposition of $\mathbb{C}[X(w)]$ into irreducible $L_w$-modules is multiplicity free and $X(w)$ is a spherical $L_w$-variety.
\end{theorem}
\begin{proof}
If $X(w)$ is smooth then $w$ is of the form
\begin{center}
$(1,\ldots,p,m+1,\ldots,m+i)$
\end{center}
for a unique choice of $0\leq p,i \leq d$ and $m<N$ such that $m+i\leq N$, $p \neq m$, and $p+i=d$.

Then $L_w=\GL_p \times \GL_{(m+i)-p} \times \GL_{N-(m+i)}$. Any $\tau:=(i_1,\ldots,i_d)$ such that $\tau \leq w$ must have that $i_1=1,\ldots,i_p=p$. Combining this with the combinatorial description of heads in Proposition \ref{proposition:HeadTypeW} we see that the single degree 1 head is $w$ itself. Thus there is a single standard degree r head $\underline{\theta}^r:=(w,\ldots,w)$, which is the sequence containing r copies of $w$. Then

\begin{center}
$\mathbb{W}_{\underline{\theta}^r}=\mathbb{W}^{(r^p)/\emptypart}\otimes \mathbb{W}^{(r^i)/\emptypart}\otimes \mathbb{W}^{\emptypart}$
\end{center}
is an irreducible $L$-module since each term in the tensor product is a Weyl module associated to a partition. Thus for all $r\geq 1$, $\mathbb{W}_{\underline{\theta}^r}^{*}$ is an irreducible $L_w$-module. Further, for $r,r' \geq 1$ with $r\neq r'$, it is clear that $\mathbb{W}_{\underline{\theta}^r}^{*}$ is not isomorphic to $\mathbb{W}_{\underline{\theta}^{r'}}^{*}$.

Combining the above with Theorem \ref{theorem:mainDecompositionTheorem} and Theorem \ref{theorem:stdMonTheoryFundamental} we have that $\mathbb{C}[X(w)]$ has a decomposition into irreducible $L_w$ modules that is multiplicity free. Thus, by Proposition \ref{proposition:coneProjVarMultFreeSphericity}, $X(w)$ is a spherical $L_w$-variety.
\end{proof}

\subsection{Determinental varieties}
\label{subsec:determinentalv}

Let $B^{-}$ be the subgroup of lower triangular matrices in $\GL_N$. Then $B^{-}[e_{id}]$ is a dense open subset of $G_{d,N}$, called \emph{the opposite big cell} in $G_{d,N}$. For a Schubert variety $X(w)$ in $\Grass_{d,N}$, let $Y(w):=B^{-}[e_{id}] \bigcap X(w)$;  note that $Y(w)$ is non-empty, since $[e_{id}]\in B^{-}[e_{id}] \bigcap X(w)$. We have that $Y(w)$ is an open affine subvariety of $X(w)$ and is usually called \emph{the opposite cell} in $X(w)$.  
We have that \begin{center}
$B^{-}[e_{id}]=\left\{ \left[ \begin{array}{cc}
Id_{d\times d} & 0_{d\times N-d} \\
X_{N-d \times d} & Id_{N-d\times N-d} \\
\end{array} \right] \in \GL_N \right\}$
\end{center}
$X_{N-d \times d}$ being a generic ${N-d \times d}$ matrix. Thus we obtain a natural identification of $M_{N-d, d}(\mathbb{C})$, the space of ${N-d \times d}$ matrices (over $\mathbb{C}$), with the dense open subset $B^{-}[e_{id}]$ of $G_{d, N}$.

\begin{definition}
\label{definition:determinentalSchubert}
Let $1\leq t<min(d,N-d)$. The \emph{determinental variety} $D_t(\mathbb{C})$ is the subset of $M_{N-d, d}(\mathbb{C})$, consisting of all $N-d \times d$ matrices over $\mathbb{C}$ with rank $\leq t$.  We have that under the above identification of $M_{N-d, d}(\mathbb{C})$ with $B^{-}[e_{id}]$,
 $D_t(\mathbb{C})$ gets identified with $Y(w)$ for $w=(t+1,\ldots,d,N-t+1,\ldots,N)$(cf. \cite[Section 1.6]{MR3308481}). The Schubert varieties with $w=(t+1,\ldots,d,N-t+1,\ldots,N)$ are referred to as \emph{determinental Schubert varieties}. Note that these are precisely the Schubert varieties which are $P_{d}\;$-stable for left multiplication.
\end{definition}

To show that the determinental Schubert varieties are spherical $L_w$-varieties we will use the following theorem which relates the number of blocks of type $L_w$, which we have denoted $\dprl{L_w}$, to the multiplicity freeness of the decomposition of $\mathbb{C}[X(w)]$.


\begin{theorem} 
\label{theorem:multFreeNumBlocks}
 Let $w=(\ell_1,\ldots,\ell_{d})\in W^{P}$. If any of the following properties hold then the decomposition of $\mathbb{C}[X(w)]$ into irreducible $L_w$-modules is multiplicity free.
\begin{enumerate}[label=(\roman*)]
\item $\dprl{L_w}=1$
\item $\dprl{L_w}=2$
\item $\dprl{L_w}=3$ and $\ell_d \neq N$
\end{enumerate}
\end{theorem}
\begin{proof}
If $\dprl{L_w}=1$, then we know that $w=(N-d+1,\ldots,N)$ and hence $X(w)$ is smooth. Thus the result follows by Theorem \ref{thm:smoothSpherical}.

Now we consider the case when $\dprl{L_w}=2$. Let $r\geq 1$. Consider an arbitrary standard degree r head $\underline{\theta} \in \He_{L_w, r}^{std}$. This proof will proceed in three steps. The first step is to show that $\mathbb{W}_{\underline{\theta}}^{*}$ is an irreducible $L_w$-module.

We have 
\begin{center}
$\mathbb{W}_{\underline{\theta}} \cong \mathbb{W}^{\nicefrac{\lambda_{\underline{\theta}}^{(1)}}{\mu_{\underline{\theta}}^{(1)}}}\otimes\mathbb{W}^{\nicefrac{\lambda_{\underline{\theta}}^{(2)}}{\mu_{\underline{\theta}}^{(2)}}}$.
\end{center}

Consider $\nicefrac{\lambda_{\underline{\theta}}^{(1)}}{\mu_{\underline{\theta}}^{(1)}}$. This skew diagram corresponds to those boxes in $\tab_{\underline{\theta}}$ whose entries are in $\Bl_{L_w, 1}$. If a row in $\tab_{\underline{\theta}}$ has a box whose entry is in $\Bl_{L_w, 1}$ then the leftmost box in the row also must have its entry in $\Bl_{L_w, 1}$ since $\tab_{\underline{\theta}}$ is a semistandard tableaux and in this case $\Bl_{L_w, 1}=\{1,\ldots,m\}$ for some $1\leq m<N$. This implies $\mu_{\underline{\theta}}^{(1)}=\emptypart$. Thus $\nicefrac{\lambda_{\underline{\theta}}^{(1)}}{\mu_{\underline{\theta}}^{(1)}}=\nicefrac{\lambda_{\underline{\theta}}^{(1)}}{\emptypart}=\lambda_{\underline{\theta}}^{(1)}$.


Next consider $\nicefrac{\lambda_{\underline{\theta}}^{(2)}}{\mu_{\underline{\theta}}^{(2)}}$.
This skew diagram corresponds to those boxes in $\tab_{\underline{\theta}}$ whose entries are in $\Bl_{L_w, 2}$. If a row in $\tab_{\underline{\theta}}$ has a box whose entry is in $\Bl_{L_w, 2}$ then the rightmost box in the row also must have its entry in $\Bl_{L_w, 2}$ since $\tab_{\underline{\theta}}$ is a semistandard tableaux and in this case $\Bl_{L_w, 2}=\{m+1,\ldots,N\}$. Thus $\nicefrac{\lambda_{\underline{\theta}}^{(2)}}{\mu_{\underline{\theta}}^{(2)}}=\nicefrac{(e^f)}{\mu_{\underline{\theta}}^{(2)}}$, for $e,f$ positive integers less than $r,d$ respectively, and we have that 
\begin{center}
$\mathbb{W}^{\nicefrac{\lambda_{\underline{\theta}}^{(2)}}{\mu_{\underline{\theta}}^{(2)}}}=\mathbb{W}^{\nicefrac{(e^f)}{\mu_{\underline{\theta}}^{(2)}}} \cong \mathbb{W}^{(\nicefrac{(e^f)}{\mu_{\underline{\theta}}^{(2)}})^{\pi}}$.
\end{center}

Hence
\begin{center}
$\mathbb{W}_{\underline{\theta}} \cong \mathbb{W}^{\lambda_{\underline{\theta}}^{(1)}}\otimes \mathbb{W}^{(\nicefrac{(e^f)}{\mu_{\underline{\theta}}^{(2)}})^{\pi}}$.
\end{center}

By Remark \ref{remark:pirotpartition} both of the terms in the tensor product are Weyl modules associated to partitions and thus $\mathbb{W}_{\underline{\theta}}$ is an irreducible $L_w$-module, which implies $\mathbb{W}_{\underline{\theta}}^{*}$ is an irreducible $L_w$-module.

Let $r,r'\geq 1$ such that $r \neq r'$. Let $\underline{\theta} \in \He_{L_w, r}^{std}$ and $\underline{\theta}' \in \He_{L_w, r'}^{std}$. Step two is to show that $\mathbb{W}_{\underline{\theta}}$ can not be isomorphic to $\mathbb{W}_{\underline{\theta}'}$, which will imply that $\mathbb{W}_{\underline{\theta}}^{*}$ can not be isomorphic to $\mathbb{W}_{\underline{\theta}'}^{*}$.

We have by step one, that
\begin{center}
$\mathbb{W}_{\underline{\theta}} \cong  \mathbb{W}^{\lambda_{\underline{\theta}}^{(1)}}\otimes \mathbb{W}^{(\nicefrac{(e^f)}{\mu_{\underline{\theta}}^{(2)}})^{\pi}}$
\end{center}
and
\begin{center}
$\mathbb{W}_{\underline{\theta}'} \cong  \mathbb{W}^{\lambda_{\underline{\theta}'}^{(1)}}\otimes \mathbb{W}^{(\nicefrac{({e'}^{f'})}{\mu_{\underline{\theta}'}^{(2)}})^{\pi}}$
\end{center}

In particular, we may count the number of boxes in each of these two $L_w$-modules.
\begin{equation}
\label{equation:detvareq1}
\begin{array}{rll}
|\lambda_{\underline{\theta}}^{(1)})| + |(\nicefrac{(e^f)}{\mu_{\underline{\theta}}^{(2)}})^{\pi}| &= |\lambda_{\underline{\theta}}^{(1)}| + \left(|\lambda_{\underline{\theta}}^{(2)}|-|\mu_{\underline{\theta}}^{(2)}|\right) & \\
 &= rd & (\textrm{by }\eqref{equation:sumofboxesallpartitions}) \\
\end{array}
\end{equation}
and
\begin{equation}
\label{equation:detvareq2}
\begin{array}{rll}
|\lambda_{\underline{\theta}'}^{(1)})| + |(\nicefrac{({e'}^{f'})}{\mu_{\underline{\theta}'}^{(2)}})^{\pi}| &= |\lambda_{\underline{\theta}'}^{(1)}| + \left(|\lambda_{\underline{\theta}'}^{(2)}|-|\mu_{\underline{\theta}'}^{(2)}|\right) & \\
 &= r'd & (\textrm{by }\eqref{equation:sumofboxesallpartitions}) \\
\end{array}
\end{equation}

Now, suppose that $\mathbb{W}_{\underline{\theta}} \cong \mathbb{W}_{\underline{\theta}'}$. This would imply that each of the individual terms in the tensor products are isomorphic. Since they are all irreducible they must have the same number of boxes. This implies
\begin{center}
$\begin{array}{rl}
|\lambda_{\underline{\theta}}^{(1)})| + |(\nicefrac{(e^f)}{\mu_{\underline{\theta}}^{(2)}})^{\pi}|&=|\lambda_{\underline{\theta}'}^{(1)})| + |(\nicefrac{({e'}^{f'})}{\mu_{\underline{\theta}'}^{(2)}})^{\pi}| \\
rd &= r'd \qquad \qquad \qquad \qquad (\textrm{by }\eqref{equation:detvareq1}\textrm{, }\eqref{equation:detvareq2})\\
\end{array}$
\end{center}

Since $d > 0$, this last equality implies that $r = r'$, which is a contradiction of our assumption. Thus we can not have $\mathbb{W}_{\underline{\theta}} \cong \mathbb{W}_{\underline{\theta}'}$ for $r \neq r'$, which further implies we can not have $\mathbb{W}_{\underline{\theta}}^{*} \cong \mathbb{W}_{\underline{\theta}'}^{*}$.

Finally, for our third step, fix an $r\geq 1$ and let  $\underline{\theta},\underline{\theta}' \in \He_{L_w, r}^{std}$. Suppose that $\mathbb{W}_{\underline{\theta}}^{*} \cong \mathbb{W}_{\underline{\theta}'}^{*}$. This implies that  $\mathbb{W}_{\underline{\theta}} \cong \mathbb{W}_{\underline{\theta}'}$ and

\begin{center}
$ \mathbb{W}^{\lambda_{\underline{\theta}}^{(1)}}\otimes \mathbb{W}^{(\nicefrac{(e^f)}{\mu_{\underline{\theta}}^{(2)}})^{\pi}} \cong \mathbb{W}^{\lambda_{\underline{\theta}'}^{(1)}}\otimes \mathbb{W}^{(\nicefrac{({e'}^{f'})}{\mu_{\underline{\theta}'}^{(2)}})^{\pi}}$
\end{center}
which implies, since these are all partitions, that  $\lambda_{\underline{\theta}}^{(1)}=\lambda_{\underline{\theta}'}^{(1)}$. Both $\tab_{\underline{\theta}}$ and $\tab_{\underline{\theta}'}$ are semistandard tableau on the diagram $(r^d)$. So $\lambda_{\underline{\theta}}^{(1)}=\lambda_{\underline{\theta}'}^{(1)}$ implies that the boxes in $\tab_{\underline{\theta}}$ whose entries are in $\Bl_{L_w, 1}$ are the same boxes in $\tab_{\underline{\theta}'}$ whose entries are in $\Bl_{L_w, 1}$. And since there are only two blocks it says the same about those boxes whose entries are in $\Bl_{L_w, 2}$. The fact that $\underline{\theta}=\underline{\theta}'$ is then a consequence of Lemma \ref{lemma:HeadClassUnique}. 

Combining the above with Theorem \ref{theorem:mainDecompositionTheorem} we have that $\mathbb{C}[X(w)]$ has a decomposition into irreducible $L_w$ modules that is multiplicity free. 

The final case we consider is when $\dprl{L_w}=3$ and $\ell_d \neq N$. In this case, if $\theta \in \He_{L_w,1}$, then $\theta \bigcap \Bl_{L_w,3}=\emptyset$ since $\theta \leq w$ and $\Bl_{L_w,3}=\{\ell_d+1,\ldots,N \}$. This implies for any $\underline{\theta} \in \He_{L_w, r}^{std}$ that $\nicefrac{\lambda_{\underline{\theta}}^{(3)}}{\mu_{\underline{\theta}}^{(3)}}=\emptypart$. But then the proof that $\mathbb{C}[X(w)]$ has a decomposition into irreducible $L_w$-modules that is multiplicity free follows by exactly the same arguments used to prove the case when $\dprl{L_w}=2$.
\end{proof}

\begin{corollary}
\label{corollary:sphericalSchubertDet}
Let $X(w)$ be a determinental Schubert variety in $\Grass_{d,N}$. Then $X(w)$ is a spherical $L_w$-variety.
\end{corollary}
\begin{proof}
By definition $w$ is of the form $(t+1,\ldots,d,N-t+1,\ldots,N)$ for $1\leq t<min(d,N-d)$. Thus $\dprl{L_w}=2$, and by Theorem \ref{theorem:multFreeNumBlocks} the decomposition of $\mathbb{C}[X(w)]$ into irreducible $L_w$-modules is multiplicity free. Thus, by Proposition \ref{proposition:coneProjVarMultFreeSphericity}, $X(w)$ is a spherical $L_w$-variety.
\end{proof}

\begin{corollary}
\label{corollary:sphericalDet}
Let $1\leq t<min(d,N-d)$ and $w=(t+1,\ldots,d,N-t+1,\ldots,N)$. Then the determinental variety $D_t(\mathbb{C})$ is $L_w$-stable and is a spherical $L_w$-variety.
\end{corollary}
\begin{proof}
As in Definition \ref{definition:determinentalSchubert}, $D_t(\mathbb{C})$ is realized as $Y(w)$ for $w=(t+1,\ldots,d,N-t+1,\ldots,N)$. By Corollary \ref{corollary:sphericalSchubertDet} we have that the determinental Schubert variety $X(w)$ is a spherical $L_w$-variety. Also, as noted in Definition \ref{definition:determinentalSchubert}, we have $Q_w=P_{d}\;$ and $L_w=\GL_d\times \GL_{N-d}$. Further, $B^-[e_{id}]$ is 
$\GL_d\times \GL_{N-d}$-stable. Hence $D_t(\mathbb{C})$ is a $L_w$-stable sub variety of $X(w)$. Thus $Y(w)$ is an open $L_w$-stable subvariety of $X(w)$, and hence is a spherical $L_w$-variety.
\end{proof}

\begin{corollary}
\label{corollary:sphericaldet}
Let $1\leq t<min(d,N-d)$ and $w=(t+1,\ldots,d,N-t+1,\ldots,N)$. Then the decomposition of $\mathbb{C}[D_t(\mathbb{C})]$ into irreducible $L_{w}$-modules is multiplicity free.
\end{corollary}
\begin{proof}
By Corollary \ref{corollary:sphericalDet}, $D_t(\mathbb{C})$ is a spherical $L_w$-variety. Further, $D_t(\mathbb{C})$ being an affine variety, the result follows from \cite[Theorem 25.1]{MR2797018}.
\end{proof}

\begin{corollary}
Let $X(w)$ be a Schubert variety in $\Grass_{2,N}$. Then $X(w)$ is a spherical $L_w$-variety.
\end{corollary}
\begin{proof}
For all $w=(\ell_1,\ell_2) \in W^{P_{2}}$, either $\dprl{L_w}=1$, $\dprl{L_w}=2$, or $\dprl{L_w}=3$ with $\ell_2 \neq N$. Thus the result follows from Theorem \ref{theorem:multFreeNumBlocks} and Proposition \ref{proposition:coneProjVarMultFreeSphericity}. 
\end{proof}

\begin{remark}
It is worth noting that some of the results in this section can be seen by more elementary arguments. The smooth Schubert variety $X(w)$, with $w=(1,\ldots,p,m+1,\ldots,m+i)$, in the Grassmannian $\Grass_{d,N}$, is isomorphic to $\Grass_{i, (m+i)-p}$. The fact that $\Grass_{i, (m+i)-p}$ is a spherical variety for the action of $\GL_{(m+i)-p}$ follows by the Bruhat decomposition; the other factors of $L_w$ act trivially. Hence a smooth Schubert variety $X(w)$ is a spherical $L_w$-variety. In the case of the determinental variety $D_{t}(\mathbb{C}) \subset M_{N-d, d}(\mathbb{C})$, the action of $L_w$ on $D_{t}(\mathbb{C})$ is given by left multiplication by matrices in $\GL_{N-d}$ and right multiplication by matrices in $\GL_{d}$. Letting $B_{N-d}$ and $B_d$ denote the Borel subgroups of upper triangular matrices in $\GL_{N-d}$ and $\GL_{d}$ respectively, $D_{t}(\mathbb{C})$ has a decomposition into $B_{N-d} \times B_{d}$ orbits given by matrix Schubert cells corresponding to partial permutations of rank at most $t$ \cite[Prop. 15.27]{MR2110098}. There are clearly a finite number of such orbits and hence $D_{t}(\mathbb{C})$ is a spherical $L_w$-variety. As far as the authors are aware, the sphericity of the determinental Schubert varieties and the Schubert varieties in $\Grass_{2,N}$ are new results.  
\end{remark}

\bibliographystyle{alpha}
\bibliography{LeviSubgroupActions.RHodges.VLakshmibai-final-rev}

\begin{thebibliography}{HTW05}

\bibitem[BLV86]{MR837530}
M.~Brion, D.~Luna, and Th. Vust.
\newblock Espaces homog\`enes sph\'eriques.
\newblock {\em Invent. Math.}, 84(3):617--632, 1986.

\bibitem[FH91]{MR1153249}
William Fulton and Joe Harris.
\newblock {\em Representation theory}, volume 129 of {\em Graduate Texts in
  Mathematics}.
\newblock Springer-Verlag, New York, 1991.
\newblock A first course, Readings in Mathematics.

\bibitem[Ful97]{MR1464693}
William Fulton.
\newblock {\em Young tableaux}, volume~35 of {\em London Mathematical Society
  Student Texts}.
\newblock Cambridge University Press, Cambridge, 1997.
\newblock With applications to representation theory and geometry.

\bibitem[Gre07]{MR2349209}
J.~A. Green.
\newblock {\em Polynomial representations of {${\rm GL}\sb {n}$}}, volume 830
  of {\em Lecture Notes in Mathematics}.
\newblock Springer, Berlin, augmented edition, 2007.
\newblock With an appendix on Schensted correspondence and Littelmann paths by
  K. Erdmann, Green and M. Schocker.

\bibitem[HTW05]{MR2115378}
Roger Howe, Eng-Chye Tan, and Jeb~F. Willenbring.
\newblock Stable branching rules for classical symmetric pairs.
\newblock {\em Trans. Amer. Math. Soc.}, 357(4):1601--1626, 2005.

\bibitem[LB15]{MR3408060}
V.~Lakshmibai and Justin Brown.
\newblock {\em The {G}rassmannian variety}, volume~42 of {\em Developments in
  Mathematics}.
\newblock Springer, New York, 2015.
\newblock Geometric and representation-theoretic aspects.

\bibitem[LMS74]{MR0354698}
V.~Lakshmibai, C.~Musili, and C.~S. Seshadri.
\newblock Cohomology of line bundles on {$G/B$}.
\newblock {\em Ann. Sci. \'Ecole Norm. Sup. (4)}, 7:89--137, 1974.
\newblock Collection of articles dedicated to Henri Cartan on the occasion of
  his 70th birthday, I.

\bibitem[LR08]{MR2388163}
Venkatramani Lakshmibai and Komaranapuram~N. Raghavan.
\newblock {\em Standard monomial theory}, volume 137 of {\em Encyclopaedia of
  Mathematical Sciences}.
\newblock Springer-Verlag, Berlin, 2008.
\newblock Invariant theoretic approach, Invariant Theory and Algebraic
  Transformation Groups, 8.

\bibitem[MS05]{MR2110098}
Ezra Miller and Bernd Sturmfels.
\newblock {\em Combinatorial commutative algebra}, volume 227 of {\em Graduate
  Texts in Mathematics}.
\newblock Springer-Verlag, New York, 2005.

\bibitem[Ses14]{MR3308481}
C.~S. Seshadri.
\newblock {\em Introduction to the theory of standard monomials}, volume~46 of
  {\em Texts and Readings in Mathematics}.
\newblock Hindustan Book Agency, New Delhi, second edition, 2014.

\bibitem[Sta99]{MR1676282}
Richard~P. Stanley.
\newblock {\em Enumerative combinatorics. {V}ol. 2}, volume~62 of {\em
  Cambridge Studies in Advanced Mathematics}.
\newblock Cambridge University Press, Cambridge, 1999.
\newblock With a foreword by Gian-Carlo Rota and appendix 1 by Sergey Fomin.

\bibitem[Tim11]{MR2797018}
Dmitry~A. Timashev.
\newblock {\em Homogeneous spaces and equivariant embeddings}, volume 138 of
  {\em Encyclopaedia of Mathematical Sciences}.
\newblock Springer, Heidelberg, 2011.
\newblock Invariant Theory and Algebraic Transformation Groups, 8.

\bibitem[Wey03]{MR1988690}
Jerzy Weyman.
\newblock {\em Cohomology of vector bundles and syzygies}, volume 149 of {\em
  Cambridge Tracts in Mathematics}.
\newblock Cambridge University Press, Cambridge, 2003.

\end{thebibliography}

\end{document}